\documentclass[oneside,english]{amsart}
\usepackage[T1]{fontenc}
\usepackage[latin9]{inputenc}
\usepackage{geometry}
\usepackage{float}
\usepackage{amstext}
\usepackage{amsthm}
\usepackage{amssymb}
\usepackage{graphicx}

\makeatletter

\providecommand{\tabularnewline}{\\}

\numberwithin{equation}{section}
\numberwithin{figure}{section}
\theoremstyle{plain}
\newtheorem{thm}{\protect\theoremname}[section]
\theoremstyle{definition}
\newtheorem{defn}[thm]{\protect\definitionname}
\theoremstyle{remark}
\newtheorem{rem}[thm]{\protect\remarkname}
\theoremstyle{definition}
\newtheorem{example}[thm]{\protect\examplename}
\theoremstyle{plain}
\newtheorem{lem}[thm]{\protect\lemmaname}
\theoremstyle{plain}
\newtheorem*{thm*}{\protect\theoremname}
\theoremstyle{plain}
\newtheorem{cor}[thm]{\protect\corollaryname}
\theoremstyle{remark}
\newtheorem{claim}[thm]{\protect\claimname}
\theoremstyle{plain}
\newtheorem{prop}[thm]{\protect\propositionname}

\makeatother

\usepackage{babel}
\providecommand{\claimname}{Claim}
\providecommand{\corollaryname}{Corollary}
\providecommand{\definitionname}{Definition}
\providecommand{\examplename}{Example}
\providecommand{\lemmaname}{Lemma}
\providecommand{\propositionname}{Proposition}
\providecommand{\remarkname}{Remark}
\providecommand{\theoremname}{Theorem}

\begin{document}
\global\long\def\bb{\mathcal{B}}%

\global\long\def\lp{\mathcal{L}}%

\global\long\def\bbg{\mathcal{B_{\gamma}}}%

\global\long\def\lpg{\mathcal{L_{\gamma}}}%

\title{Moment Summation Methods and Non-Homogeneous Carleman Classes}
\author{Avner Kiro}
\date{}
\address{A. Kiro, Faculty of Mathematics and Computer Science,
The Weizmann Institute of Science,
234 Herzl Street,Rehovot 76100, Israel}
\email{avner-ephraiem.kiro@weizmann.ac.il}
\begin{abstract}
We extend the classical theorems of F. Nevanlinna and Beurling by characterizing the image of various spaces of smooth functions under the generalized Laplace transform. To achieve this, we introduce and analyze novel non-homogeneous Carleman classes, which generalize the traditional homogeneous definitions. This characterization allows us to derive necessary and sufficient conditions for the applicability of moment summation methods within a given class of functions. Furthermore, we establish an extension of \'{E}calle's concept of quasianalytic continuation and apply these results to the theory of multi-summability and Euler-type differential equations.
\end{abstract}

\maketitle
\tableofcontents{}

\section{Introduction}

Moment summation methods (also known as generalized Borel-Laplace
summation) associate functions with formal power series. They extend the classical notion of the convergent power series. Following \cite{hardy2000divergent}, we briefly recall some
of the basic notions related to these methods. Let $\left(\gamma_{n}\right)_{n\geq1}$
be a fast-growing sequence of positive numbers, with  $\lim_{n\to\infty}\gamma_{n}^{1/n}=\infty$.
Suppose that it is a moment sequence for a function $K$ defined
on $\mathbb{R}_{+}$, that is 
\[
\int_{0}^{\infty}t^{n}K(t)dt=\gamma_{n+1},\quad\forall n\in\mathbb{Z}_{+}.
\]
Let $\widetilde{f}\sim\sum_{n\geq0}a_{n}x^{n}$ be a formal power
series. Suppose the series $F(t)=\sum_{n\geq0}\frac{a_{n}}{\gamma_{n+1}}t^{n}$
converges for $0\leq t\leq t_{0}$. If it has an analytic continuation
to the whole positive ray $\mathbb{R}_{+}$ such that
\[
f(x)=\int_{0}^{\infty}F(xt)K(t)dt,
\]
then $\widetilde{f}$ is said to be\emph{ $\left(\gamma\right)$--summable}
to the function $f$. We define \emph{the generalized Borel and Laplace
transforms} by $\mathcal{B_{\gamma}}f=F$ and $\mathcal{L}_{\gamma}F=f$,
respectively. We say that\emph{ $\left(\gamma\right)$--}summation
is \emph{applicable} in a class $D$  of $C^{\infty}$ smooth
functions if $\mathcal{L}_{\gamma}\mathcal{B}_{\gamma}f=f$ for every $f\in D$.

The generalized Borel transform $\mathcal{B}_{\gamma}$  depends only on the Taylor coefficients of $f$ at the origin. For $\left(\gamma\right)$--summation to apply in a class $D$ of smooth functions, the Borel map
\[
f\mapsto\left(f^{(n)}(0)\right)_{n\geq0}
\]
must be  injective in $D$. To capture this, we introduce the following definition.
\begin{defn}
Let $I$ be an interval, and let $D\subseteq C^\infty(I)$. We say that the
class $D$ is \emph{quasianalytic} if, for any $f_{1},f_{2}\in D$ satisfying 
\[
f_{1}^{(n)}\left(c\right)=f_{2}^{(n)}(c)
\]
for all $n\in\mathbb{Z}_{+}$ at some fixed point $c\in I$, we have $f_{1}\equiv f_{2}$
in $I$.
\end{defn}

The entire function 
\[
E(z):=\sum_{n\geq0}\frac{z^{n}}{\gamma_{n+1}}
\]
plays a key role in moment summation theory. Observe
that $E$ is the generalized Borel transform $\mathcal{B}_{\gamma}$
of the Cauchy kernel $x\mapsto\frac{1}{1-x}$, the identity in the ring of formal power series (under term-wise addition and multiplication).  For example, in the classical Borel--Laplace summation, $\gamma_{n+1} = n!$, $K(t) = e^{-t}$, and $E(z) = e^z$. In the logarithmic case $\gamma_n = \log^n(n + e)$, considered by Beurling \cite{beurling1989collected}, we have
\[
\log E(t)\sim-\log K(t)\sim\frac{e^{t}}{t}, \quad t\to\infty.
\]
We emphasize that this matching between $E$'s growth and $K$'s decay is essential and will feature prominently below.

Hereafter, we restrict to functions $F$ such
that $\mathcal{L}_{\gamma}F\in C^{\infty}$ in some right neighborhood
of $0$. Under regularity conditions like those imposed later, a necessary and sufficient condition for $\mathcal{L}_{\gamma}F$ to be a $C^{\infty}$ function on the interval $\left[0,1\right)$ is that, for any $\eta>1$ and $n\in\mathbb{Z}_{+}$,
\begin{equation}
\left|F^{(n)}(x)\right|=O\left(E\left(\eta x\right)\right),\quad\text{as}\quad x\to\infty.\label{eq:necessary conditions for C^infty}
\end{equation}
See sections \ref{subsec:1st inc}
and \ref{subsec:2nd inc} for details.

The purpose of this text is to describe the image under $\mathcal{L}_{\gamma}$
of classes of smooth functions. These classes arise by bounding the implicit constants in the above estimate. 
\begin{defn}
Let $\gamma$ be as before, $M=\left(M_{n}\right)_{n\geq0}$ be a
sequence of positive numbers, and $I=\left[0,a\right)$ (with $0<a\leq+\infty)$.
The class $A\left(M,\gamma;I\right)$ consists of all functions $F\in C^{\infty}\left[0,\infty\right)$
such that, for any $\eta>1$, there exists a constant $C=C(\eta,F)>0$
with
\[
\left|F^{(n)}\left(x\right)\right|\leq C^{n+1}M_{n}E\left(\eta\cdot\frac{x}{a}\right).
\]
We also define
\[
A\left(M,\gamma;0^{+}\right)=\bigcap_{a>0}A\left(M,\gamma;[0,a)\right)
\]
\end{defn}

\noindent Note that, by the Denjoy--Carleman Theorem \cite{carleman1926fonctions},
if $M$ is a log-convex sequence, then $A\left(M,\gamma;I\right)$
is quasianalytic if and only if 
\[
\sum_{n\geq0}\frac{M_{n}}{M_{n+1}}=\infty.
\]
For the particular choice $M_{n}=n!,$ the functions in the class
$A\left(M,\gamma;I\right)$ are analytic in some half-strip $\left\{ z:\text{dist}\left(z,\mathbb{R}_{+}\right)\leq\varepsilon\right\} .$

Using the classes $A\left(M,\gamma;I\right)$, we extend the notion
of \emph{ $\left(\gamma\right)$--}summation applicability. Namely, it is 
\emph{applicable }in a class $D\subseteq C^{\infty}$  if there exists a quasianalytic
$A\left(M,\gamma;I\right)$ such that all the functions in $\mathcal{L}_{\gamma}D$
are restrictions of functions in $A\left(M,\gamma;I\right)$ to a
right neighborhood of the origin. This extends \'{E}calle's notion of quasianalytic continuation \cite{ecalle1993cohesive}.

\subsection{Smoothness of the generalized Laplace transform}

The generalized Laplace transform,
\[
f(x)=\left(\mathcal{L}_{\gamma}F\right)(x)=\int_{0}^{\infty}F(xt)K(t)dt,
\]
can be viewed as a multiplicative convolution of $F$
with
$$K_{1}\left(u\right)\mathrel{\mathop{:}}=\frac{1}{u}K\left(\frac{1}{u}\right).$$
Indeed, with the change of variables $u=\frac{1}{t}$,
\begin{align*}
f(x) & =\int_{0}^{\infty}F\left(\frac{x}{u}\right)\frac{1}{u}K\left(\frac{1}{u}\right)\frac{du}{u}\\
& =\int_{0}^{\infty}F\left(\frac{x}{u}\right)K_{1}\left(u\right)\frac{du}{u}=\left(F\star K_{1}\right)(x).
\end{align*}
Viewed as such a convolution, $f$ ``inherits'' smoothness from both $F$ and $K$. Writing the derivatives under the integral sign gives, for $x\ge 0$ and $n\ge 0$,

\begin{equation}
f^{(n)}(x)=\int_{0}^{\infty}F^{(n)}\left(xt\right)t^{n}K(t)dt,\label{eq:deriv of convolution equ}
\end{equation}
and under the admissibility hypotheses this yields the near-origin bound
\begin{equation}
\left|f^{(n)}(x)\right|\leq C^{n+1}M_{n},\quad\forall x\in[0,\delta],\,n\geq0,\label{eq: Carleman A type bdd}
\end{equation}
for sufficiently small $\delta>0$.

It is also useful to switch to logarithmic coordinates. Setting $x=e^{\xi}$ and
$u=e^{\xi-v}$ we obtain
\[
f(e^\xi)=\int_{0}^{\infty} F(e^{v})\,K_1(e^{\xi-v})\,dv,
\]
hence, by differentiating in $\xi$ and using the chain rule,
\[
\frac{d^{n}}{d\xi^{n}}\,f(e^\xi)
=\int_{0}^{\infty} F(e^{v})\,\frac{d^{n}}{d\xi^{n}}\,K_1(e^{\xi-v})\,dv.
\]

From this one obtains the logarithmic bound (see \S\ref{subsec:1st inc})
\begin{equation}
\left|\frac{d^{n}}{d\xi^{n}}f\left(e^{\xi}\right)\right|\leq C_{\eta,\delta}^{\prime}\eta^{n}\widehat{\gamma}_{n},\quad\xi\in\left(-\infty,\log\delta\right],\,n\geq0,\,\eta>1,\label{eq: non-homog Carleman bdd}
\end{equation}
where $\widehat{\gamma}$ is given by 
\[
\widehat{\gamma}_{n}:=\left(\frac{\gamma_{n+1}}{\gamma_{n}}\right)^{\frac{\pi}{2}n}.
\]
The estimates \eqref{eq: Carleman A type bdd} and \eqref{eq: non-homog Carleman bdd} are the basic input for the local classes introduced below.

\begin{defn}
Let $I\subseteq\left[0,\infty\right)$ be an interval, $\eta>0$ and
$M,\,N$ be two positive sequences. The class $B_{\eta}(M,N;I)$ consists
of all functions $f\in C^{\infty}\left(I\right)$ such that for any
compact $J\subset I$, there exists a constant $C=C\left(f,J\right)>0$
such that the following two conditions hold:
\begin{enumerate}
\item ${\displaystyle \left|f^{(n)}(x)\right|\leq C^{n+1}M_{n},}$for any
$x\in J$ and $n\geq0;$
\item ${\displaystyle \left|\frac{d^{n}}{d\xi^{n}}f\left(e^{\xi}\right)\right|\leq C\eta^{n}N_{n},}$for
any $\xi$ so that $e^{\xi}\in J$ and $n\geq0.$
\end{enumerate}
We also set
\[
B_{\eta}(M,N;0^{+})=\bigcap_{a>0}B_{\eta}(M,N;\left[0,a\right))
\]

\end{defn}

\subsection{Main results}

\subsubsection{A local result}

We begin by presenting our main result for right germs at the origin.
\begin{thm}
\label{thm:Main thm}Let $\gamma$ be an admissible weight and let
$M$ be a regular sequence. For any $\eta>1,$ we have 
\[
\mathcal{B_{\gamma}}B_{1/\eta}(M\gamma,\widehat{\gamma};0_{+})\subseteq A(M,\gamma;0_{+})\quad\text{and}\quad\lpg A(M,\gamma;0_{+})\subseteq B_{\eta}(M\gamma,\widehat{\gamma};0_{+}).
\]
Moreover, if the class $A(M,\gamma;0_{+})$ is quasianalytic (i.e.,
if $\sum M_{n}/M_{n+1}=\infty)$, then 
\[
B_{1/\eta}(M\gamma,\widehat{\gamma};0_{+})\subseteq\lpg A(M,\gamma;0_{+})\subseteq B_{\eta}(M\gamma,\widehat{\gamma};0_{+}),
\]
and the $(\gamma)$--summation is applicable in the class $B_{1/\eta}(M\gamma,\widehat{\gamma};0_{+})$. 
\end{thm}

The precise definitions of admissible weights and regular sequences
are given in Sections \ref{subsec:Admissible-weights-and} and \ref{subsec:Carleman-type-classes}.
The following table demonstrates the relations between admissible
weights $\gamma$ and the corresponding sequence $\widehat{\gamma}$:

\begin{table}[H]
\begin{tabular}{|c|c|c|}
\hline 
$\gamma$ & $\widehat{\gamma}$ & parameters\tabularnewline
\hline 
\hline 
$\log^{\alpha n}n\cdot\log^{\beta n}\log n$ & $n!\left(\frac{2}{\pi\alpha}\log n\right)^{n}$ & $\alpha>0,\,\beta\in\mathbb{R}$\tabularnewline
\hline 
$\log^{\beta n}\log n$ & $n!\left(\frac{2}{\pi\beta}\log n\cdot\log\log n\right)^{n}$ & $\beta>0$\tabularnewline
\hline 
$\exp(n\log^{\alpha}n)$ & $n!\left(\frac{2}{\pi\alpha}\log^{1-\alpha}n\right)^{n}$ & $0<\alpha<1$\tabularnewline
\hline 
$\exp\left(\frac{n\log n}{\log^{\alpha}\log n}\right)$ & $n!\left(\frac{2}{\pi\alpha}\log\log n\right)^{n}$ & $\alpha>0$\tabularnewline
\hline 
$\Gamma\left(1+\frac{n}{\alpha}\right)$ & $n!\left(\frac{2}{\pi}\alpha\right)^{n}$ & $\alpha>0$\tabularnewline
\hline 
\end{tabular}

$ $

\caption{$\gamma$ and $\widehat{\gamma}$ correspondence.}
\end{table}
We also mention that the following sequences are regular:
\[
M_{n}=n!^{\text{\ensuremath{\alpha_{1}}}}\Gamma\left(1+\alpha_{2}n\right)\cdot\exp\left(\alpha_{3}\frac{n\log n}{\log^{\alpha_{3}}\log n}+n\log^{\alpha_{4}}n\right)\log^{\alpha_{5}n}n\cdot\log^{\alpha_{6}n}\left(\log n\right)
\]
for any $\alpha_{1},\cdots,\alpha_{6}\geq0$.
\begin{rem}
Note that when $A(M,\gamma;0_{+})$ is non-quasianalytic, the inclusion
$\mathcal{B_{\gamma}}B_{1/\eta}(M\gamma,\widehat{\gamma};0_{+})\subseteq A(M,\gamma;0_{+})$
should be understood at the level of formal power series, i.e., for
any $f\in B_{1/\eta}(M\gamma,\widehat{\gamma};0_{+})$ there exists
a function $F\in A(M,\gamma;0_{+})$ so that
\[
\frac{f^{(n)}(0)}{\gamma_{n+1}}=F^{(n)}(0),\quad\forall n\in\mathbb{Z}_{+}.
\]
\end{rem}

\subsubsection{A global result in the non-analytic case}

In what follows, we will say that a sequence $N$ is \emph{analytic}
if
\[
\limsup_{n\to\infty}\left(\frac{N_{n}}{n!}\right)^{1/n}<\infty,
\]
otherwise we will say that $N$ is non-analytic.

Our local result can be strengthened in different ways whenever the
sequence $\widehat{\gamma}$ is analytic or non-analytic. We first
present a version for the non-analytic case. 
\begin{thm}
\label{thm:Main thm non analytic}Suppose that $\gamma$ is an admissible
weight with non-analytic $\widehat{\gamma}$, $M$ is a regular sequence, $\eta>1$ and $I=\left[0,a\right)$ (with $0<a\leq+\infty)$.
Then, 
\[
\mathcal{B_{\gamma}}B_{1/\eta}(M\gamma,\widehat{\gamma};I)\subseteq A(M,\gamma;I),\quad\lpg A(M,\gamma;I)\subseteq B_{\eta}(M\gamma,\widehat{\gamma};I).
\]
If, in addition, the class $A(M,\gamma;I)$ is quasianalytic, then
\[
B_{1/\eta}(M\gamma,\widehat{\gamma};I)\subseteq\lpg A(M,\gamma;I)\subseteq B_{\eta}(M\gamma,\widehat{\gamma};I),
\]
and the $(\gamma)$--summation is applicable in the class $B_{1/\eta}(M\gamma,\widehat{\gamma};I).$
\end{thm}

\subsubsection{A global result in the analytic case}

Our result for the analytic case requires the following definition.
\begin{defn}
Suppose that $\gamma$ is an admissible weight with analytic $\widehat{\gamma}$,
$I=[0,a)$ and $\eta>0$. Put

\[
\Omega_{\eta}\left(\gamma\right):=\left\{ z\in\mathbb{C}\setminus\left\{ 0\right\} :\sup_{t>0}E(t\eta)\left|K\left(\frac{t}{z}\right)\right|<\infty\right\} .
\]
The class $D\left(M,\gamma;I\right)$ consists of all functions $f$
such that for any $\eta>a$, $f\in\text{Hol}\left(\Omega_{\eta}\left(\gamma\right)\right)\cap C^{\infty}\left(\left\{ 0\right\} \cup\Omega_{\eta}\left(\gamma\right)\right)$
and there exists a constant $C=C_{\eta}>0$ such that
\[
\left|R_{n}(z,f)\right|\leq C^{n+1}\frac{M_{n}\gamma_{n+1}}{n!}|z|^{n},\quad\forall n\in\mathbb{Z}_{+},\,\forall z\in\Omega_{\eta}\left(\gamma\right),
\]
where $R_{n}\left(z,f\right)$ is the $n-$th order remainder of the
Taylor series of $f$ around the origin, i.e., 
\[
R_{n}(z,f):=f(z)-\sum_{0\leq k<n}\frac{f^{(k)}(0)}{k!}z^{k}.
\]

\end{defn}

\begin{example}
\label{exa: Moroz Nev}For $\alpha>\frac{1}{2}$, consider $\gamma_{n}=\Gamma\left(1+\frac{n}{\alpha}\right)$,
where $\Gamma$ is the Euler $\Gamma$-function. We have $K\left(t\right)=\exp\left(-t^{\alpha}\right)$
and by \cite[Section 3.5.3]{goldberg2008value} we also have 
\[
E(z)=\begin{cases}
\alpha\exp(z^{\alpha})+O(1/z), & |\arg z|\leq\frac{\pi}{2\alpha};\\
O(1/z), & \frac{\pi}{2\alpha}\le|\arg z|\leq\pi.
\end{cases}
\]
For any $\eta_{1}<\eta<\eta_{2},$ the set $\Omega_{\eta}\left(\gamma\right)$
satisfies (see Figure 1.1)
\[
\left\{ z:\text{Re\ensuremath{\frac{1}{z^{\alpha}}\ge\eta_{2}^{\alpha}}}\right\} \subset\Omega_{\eta}\left(\gamma\right)\subset\left\{ z:\text{Re\ensuremath{\frac{1}{z^{\alpha}}\ge\eta_{1}^{\alpha}}}\right\} 
\]

\begin{figure}[H]
\includegraphics[scale=0.4]{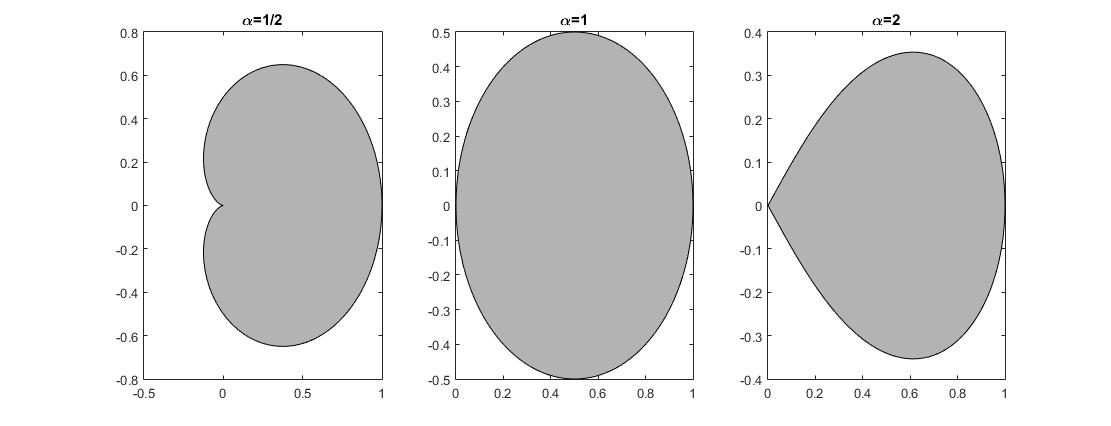}\label{fig:3 domain}\caption{$\text{\ensuremath{\Omega_{1}\left(\gamma\right)} for \ensuremath{\alpha=\frac{1}{2},\,1,\,2}}$}
\end{figure}
The next lemma demonstrates that the classes $D\left(M,\gamma;I\right)$
are also closely related to non-homogeneous Carleman classes.
\end{example}

\begin{lem}
If $\gamma$ is a weight with analytic $\widehat{\gamma}$, and $M$
is a regular sequence, then for any $\eta>0$ 
\[
B_{1/\eta}(M\gamma,\widehat{\gamma};0_{+})\subseteq\bigcup_{\delta>0}D\left(M,\gamma;[0,\delta)\right)\subseteq B_{\eta}(M\gamma,\widehat{\gamma};0_{+}).
\]
\end{lem}

\begin{thm}
\label{thm:Main thm analytic}Suppose that $\gamma$ is an admissible
weight with analytic $\widehat{\gamma}$, $M$ is a regular sequence,
$I=[0,a)$ and $\eta>1$. We have
\[
\bbg D\left(M,\gamma;I\right)\subseteq A\left(M,\gamma;I\right)\quad\text{and \ensuremath{\quad}}\lpg A\left(M,\gamma;I\right)\subset D\left(M,\gamma;I\right).
\]
Moreover, if the class $A\left(M,\gamma;I\right)$ is quasianalytic,
then $D\left(M,\gamma;I\right)$ is quasianalytic, $\lpg$$:A\left(M,\gamma;I\right)\to D\left(M,\gamma;I\right)$
is a \emph{bijection} with inverse $\bbg$, and
\[
\left(\bbg f\right)(z)=\frac{1}{2\pi i}\int_{\partial\Omega_{\eta}\left(\gamma\right)}f(w)E\left(\frac{z}{w}\right)\frac{dw}{w},\quad\forall z\in\Omega_{\eta}\left(\gamma\right).
\]
\end{thm}

\subsection{F. Nevanlinna and Beurling theorems\label{subsec:F.-Nevanlinna-and}}

One natural domain of definition for the generalized Laplace transform
$\mathcal{L}_{\gamma}$, which was considered by many authors \cite{nevanlinna1918theorie,beurling1989collected},
is the space of all functions $F$, analytic in some half strip $\left\{ z:\left|\text{Im}(z)\right|\le C^{-1},\,\text{Re}\left(z\right)\geq0\right\} $
and satisfying therein
\[
\left|F(z)\right|=O\left(E(\eta|z|)\right),\quad\forall\eta>1.
\]
Note that by Cauchy estimates for derivatives of analytic functions,
this space coincides with $A\left(n!,\gamma;\left[0,1\right)\right).$

Next, we compare two classical descriptions of $\mathcal{L}_{\gamma}A\left(n!,\gamma;\left[0,1\right)\right)$ in the cases 
$\gamma_{n+1}=n!$ and  $\gamma_{n+1}=\log^{n}\left(n+e\right).$
These results (and their extensions) were among our motivations for this work.
We mention that both of these results and their extensions mentioned
below are particular cases of Theorems \ref{thm:Main thm non analytic}
and \ref{thm:Main thm analytic}.
\begin{thm*}[F. Nevanlinna \cite{nevanlinna1918theorie}, see also Sokal \cite{sokal1980improvement}]
Set $\gamma_{n+1}=n!.$ The Laplace transform $\mathcal{L}_{\gamma}$
maps the space $A\left(n!,\gamma;\left[0,1\right)\right)$ bijectively
onto the subspace of $C^{\infty}\left[0,1\right)$ consisting of all
functions, $f,$ such that for any $\eta>1$, the function $f$ is
holomorphic in the disk $\Omega_{\eta}=\left\{ z:|2\eta z-1| <1\right\},$
and there exists a constant $C=C_{\eta}>0$ such that

\begin{equation}
\left|f^{\left(n\right)}\left(z\right)\right|\leq C^{n+1}n!\gamma_{n+1},\quad\forall n\in\mathbb{Z}_{+},\,\forall z\in\overline{\Omega_{\eta}}.\label{eq:Carleman type inq Navanlinna}
\end{equation}
Moreover, for any such $f$, $\mathcal{L}_{\gamma}\mathcal{B}_{\gamma}f=f$
in $\left[0,1\right)$.
\end{thm*}
\begin{figure}[H]
\includegraphics{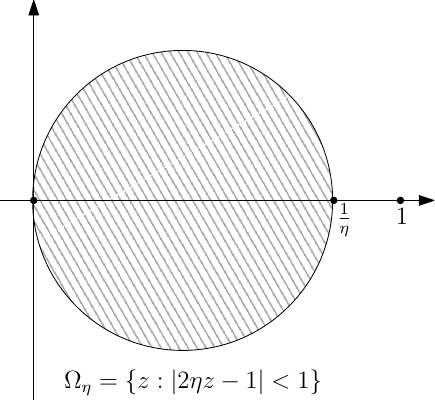}\caption{The disk $\Omega_{\eta}$}
\end{figure}
\begin{thm*}[{Beurling \cite[pp. 420--429]{beurling1989collected} }]
Set $\gamma_{n+1}=\log^{n}\left(n+e\right).$ The generalized Laplace
transform $\mathcal{L}_{\gamma}$ maps the space $A\left(n!,\gamma;\left[0,1\right)\right)$
bijectively onto the subspace of $C^{\infty}\left[0,1\right)$ consisting
of all functions, $f,$ such that for any $\eta>1$ there exists a
constant $C=C_{\eta}>0$ such that
\begin{equation}
\left|f^{\left(n\right)}\left(x\right)\right|\leq C^{n+1}n!\gamma_{n+1},\quad\forall n\in\mathbb{Z}_{+},\,\forall0\leq x\leq\frac{1}{\eta}.\label{eq:Careleman type inq Beurling}
\end{equation}
Moreover, for any such $f$, $\mathcal{L}_{\gamma}\mathcal{B}_{\gamma}f=f$
in $\left[0,1\right)$.
\end{thm*}
Although inequalities (\ref{eq:Carleman type inq Navanlinna}) and
(\ref{eq:Careleman type inq Beurling}) are identical, the main difference
between $\mathcal{L}_{\gamma}A\left(\gamma\right)$ in the Nevanlinna
and Beurling cases is the extra analyticity in the former case.

For the sequence $\gamma_{n+1}=\Gamma\left(1+\tfrac{n}{\beta}\right)$,
with $\beta>\frac{1}{2}$, partial results in the spirit of Nevanlinna's
theorem were obtained by Moroz \cite{moroz1990summability}. In this
case, the domain of analyticity should be modified from the disc $\Omega_{\eta}$
to a sector with a vertex at the origin and opening $\pi/\beta$ (see
Example \ref{exa: Moroz Nev} and Figure 1.1, c.f. discusion in  \cite[\S4.5]{costin}).

An extension of Beurling's theorem was obtained by the author \cite[Theorems 1 and 3]{kiro2018taylor},
where it was shown that the same conclusion is true for admissible weights
$\gamma$, if and only if $\gamma_{n}^{1/n}=O\left(\log n\right)$
as $n\to\infty.$ This logarithmic threshold was another motivation
for this work (see discussion in \ref{subsec: log th}).

Next, we demonstrate how our result reads in the Nevanlinna case
when $\gamma_{n+1}=n!.$ 
\begin{cor}
Set $\gamma_{n+1}=n!$ and let $M$ be a regular sequence. The Laplace
transform $\mathcal{L}_{\gamma}$ maps the space $A\left(M,\gamma;\left[0,1\right)\right)$
bijectively onto the subspace of $C^{\infty}\left[0,1\right)$ consisting
of all functions, $f,$ such that for any $\eta>1$, the function
$f$ is holomorphic in the disk $\Omega_{\eta}=\left\{ z:|2\eta z-1|\right\} <1,$
and there exists a constant $C=C_{\eta}>0$ such that
\[
\left|f^{\left(n\right)}\left(z\right)\right|\leq C^{n+1}M_{n}\gamma_{n+1},\quad\forall n\in\mathbb{Z}_{+},\,\forall z\in\overline{\Omega_{\eta}}
\]
Moreover, for any such $f$, $\mathcal{L}_{\gamma}\mathcal{B}_{\gamma}f=f$
in $\left[0,1\right)$.
\end{cor}

\begin{rem}
Beurling's result and its extensions are originally given for functions
$f$ in defined on an open interval that contains the origin. The version
presented above for the half open interval $[0,1)$ is new and follows
from Theorem \ref{thm:Main thm non analytic}.
\end{rem}

\subsubsection{Structure of this work}

The rest of this work is organized as follows. In the next section,
we define all the classes of smooth functions that will be used in
this work and present some auxiliary results regarding the asymptotic
behavior of the functions $K$ and $E.$ Section 3 is devoted to Non-homogeneous
Carleman classes and quasianalyticity. In Section 4 we prove bounds
of the form (\ref{eq: Carleman A type bdd}) and (\ref{eq: non-homog Carleman bdd})
for the generalized Laplace transform. Section 5 is devoted to extensions
of Dynkin's representation of functions in Carleman classes to the
non-homogeneous case. Section 6 is devoted to bounds on the
derivatives of the generalized Borel transform of functions belonging
to non-homogeneous Carleman classes. Finally, Section 7 is devoted
to applications including ones related to Euler-type differential equations, multi-summability, and resurgent theory.

\section{Basic notions}

\subsection{Notation}

We shall use the symbol  $C$ to denote   large positive constants which may change their values in different occurrences. If a constant $C$ depends on some  parameter $p$, we will write $C_p$ (again the value can change in different   occurrences). Given two functions $f$ and $g$ with the same domain of definition, we write $f\lesssim g$ whenever $f(x)\leq Cg(x)$ for some constant $C$. If the constant $C$ in the last inequality   depends on some parameter $p$, we will write $f\lesssim_p g$. We use the notation $f\asymp g$ ($f\asymp_p g$) if $f\lesssim g$ and $g\lesssim f$ ($f\lesssim_p g$ and $g\lesssim_p f$). If for some set $\Pi$ we have $f\vert_\Pi \lesssim g\vert_\Pi$, we will write $f\lesssim g$ on $\Pi$, and we will do the same for $\asymp$, $\lesssim_p$ and $\asymp_p$. If $\lim_{x\to\infty} \frac{f(x)}{g(x)}=1$, we will write $f\sim g$. For a function $f$ defined on the positive ray, we will say that a property $P$ is eventually satisfied if there exists $\rho_0$ such that  $P$  is  satisfied in the interval $[\rho_0,\infty)$. Given two sequences $M$ and $N$, we will write $M\approx N$ if $M_n^{1/n}\sim N_n^{1/n}$.

\subsection{Admissible weights and the sequence $\widehat{\gamma}$\label{subsec:Admissible-weights-and}}

Hereafter, we consider the sequence $\left(\gamma_{n}\right)_{n\in\mathbb{N}}$
as the values of a function $\gamma$ on $\mathbb{N}$ (i.e., $\gamma_{n}=\text{\ensuremath{\gamma(n)}})$.
We assume that the function $\gamma$ is analytic and non-vanishing
in the angle $\left\{ s:\left|\arg(s+c)<\alpha_{0}\right|\right\} $
with $\frac{\pi}{2}<\alpha_{0}$ and $c>0,$ and is positive on $(-c,\infty).$
We put 
\[
L(s)=\gamma(s)^{1/s}\quad\text{and}\quad\varepsilon(s)=s\frac{L^{\prime}(s)}{L(s)}.
\]

\begin{defn}
We call the function $\gamma$ admissible if $\varepsilon$ is positive
and bounded on $\mathbb{R}_{+}$, and satisfies the following conditions:
\end{defn}

\begin{itemize}
\item[(A)]  $\int^{\infty}\frac{\varepsilon(\rho)}{\rho}\mathrm{d}\rho=\infty,$
\item[(B)]  $\varepsilon(\lambda\rho)\sim\varepsilon(\rho)$ as $\rho\to\infty$
locally uniformly for $\lambda\in\left\{ s:\left|\arg(s)<\alpha_{0}\right|\right\} ,$
\item[(C)] $\lim_{\rho\to\infty}\varepsilon(r)$ exists and is smaller than $2.$
\item[(D)] If $\lim_{\rho\to\infty}\varepsilon(\rho)=0$, we assume that $\varepsilon$
is eventually decreasing and $\left|\varepsilon^{\prime}(\rho)\right|\gtrsim_{\delta}e^{-\delta\rho}$,
for any $\delta>0$.
\end{itemize}
It is not hard to show that under these assumptions we have 
\begin{equation}
\log\left|\gamma\left(\rho e^{i\theta}\right)\right|=\cos\theta\log L(\rho)+\theta\sin\theta\varepsilon\left(\rho\right)+o\left(\varepsilon\left(\rho\right)\right),\label{eq:gamma asymp full}
\end{equation}
locally uniformly for $\theta\in\left(-\alpha_{0},\alpha_{0}\right)$
as $\rho\to\infty.$ Given an admissible $\gamma$, we put 
\[
\widehat{\gamma}_{n}=\sup_{\rho>0}\rho^{n}\left|\gamma\left(i\rho\right)\right|=\sup_{\rho>0}\rho^{n}\exp\left(-\left(1+o(1)\right)\frac{\pi}{2}\rho\varepsilon(\rho)\right)
\]
The sequence $\widehat{\gamma}_{n}$ occurs naturally while estimating
the term $\frac{d^{n}}{dt^{n}}\left(t^{n}K(t)\right)$ that appears
in (\ref{eq:deriv of convolution equ}) (see Lemma \ref{lem: K_1 deriv estimate}
and its proof).

\subsection{The functions $K$ and $E$}

In this section we present results from \cite{KIROSODIN} regarding
the asymptotic behavior of the functions $K$ and $E.$

\subsubsection{The saddle point equation}

The asymptotics of the functions $K$ and $E$ for large $z$ are
determined by the saddle-point of the function $s\mapsto\log\gamma(s)-s\log z=s\log L(s)-s\log z$,
that is, by the equation
\begin{equation}
\log L(s)+s\frac{L'(s)}{L(s)}=\log z.\label{eq:saddlePoint}
\end{equation}

For $0<\alpha<\alpha_{0}$ and $\rho_{0}>0$, put
\[
S(\alpha,\rho_{0})=\{s:\;|\arg(s)|<\alpha,\;|s|>\rho_{0}\}.
\]
Then, it is not difficult to show that under our admissibility assumption,
the left-hand side of the saddle-point equation (\ref{eq:saddlePoint}) is a
univalent function in $S(\alpha,\rho_{0})$ (see \cite[Section 1.3]{KIROSODIN}).
From here on, we assume that this is the case, and put
\[
\Omega(\alpha)=\left\{ z\colon\log z=\log L(s)+s\frac{L'(s)}{L(s)},\,s\in S(\alpha,\rho_{0})\right\} .
\]
In general, this is a domain in the Riemann surface of $\log z$,
but by choosing $\rho_{0}$ sufficiently large, we can treat it as
a subdomain of the slit plane $\mathbb{C}\setminus\mathbb{R}_{-}$,
provided that
\[
\limsup_{\rho\to\infty}\varepsilon(\rho)<\frac{\pi}{\alpha},
\]

In what follows, we denote by $s_{z}=\rho_{z}e^{{\rm i}\theta_{z}}$
the unique solution of the saddle-point equation (\ref{eq:saddlePoint}).

\subsubsection{Asymptotic behavior of the functions $K$ and $E$\label{subsec:Asymptotic-behavior-of K and E}}
\begin{thm*}[A]
Suppose that $\gamma$ is admissible. Then, for any $\delta>0$,
the function $K$ is analytic in $\Omega(\alpha_{0}-\delta)$ and
\[
K(z)=\left(1+o(1)\right)\sqrt{\frac{s}{2\pi\varepsilon(s)}}\exp\left(-s\varepsilon(s)\right),\quad z\to\infty,
\]
uniformly in $\Omega(\alpha_{0}-\delta)$. Here $s=s_{z}$ and the
branch of the square root is positive on the positive half-line.
\end{thm*}
\begin{thm*}[B]
Suppose that $\gamma$ is admissible. Then, given a sufficiently
small $\delta>0$, we have
\[
zE(z)=\left(1+o(1)\right)\sqrt{2\pi\frac{s}{\varepsilon(s)}}\exp\left(s\varepsilon(s)\right)+o(1),\quad z\to\infty,
\]
uniformly in $\Omega(\tfrac{\pi}{2}+\delta),$ and
\[
zE(z)=O(1),\quad z\to\infty
\]
uniformly in $\mathbb{C}\setminus\Omega(\tfrac{\pi}{2}+\delta).$
Here $s=s_{z}$ and the branch of the square root is positive on the
positive half-line.
\end{thm*}

\subsection{Carleman type classes\label{subsec:Carleman-type-classes}}

Throughout this work, we will work with the two-point compactification
of the real line $[-\infty,\infty].$ Given a positive sequence $M$,
a closed interval $J\subseteq[-\infty,\infty],$ and a positive constant
$\mu>0$, the\emph{ Carleman type} class $C_{\mu}^{M}(J)$ consists
of all functions $f\in C^{\infty}(J)$ so that
\[
\left|f^{(n)}(x)\right|\leq C(f)\cdot\mu^{n}M_{n},\quad\forall x\in J\cap\mathbb{R},\,n\in\mathbb{Z}_{\geq0}.
\]
Given an arbitrary interval $I$, put 
\[
C_{\mu}^{M}(I)=\bigcup_{J\subseteq I}C_{\mu}^{M}(J),
\]
where the union is over all closed sub-intervals $J\subseteq I$.
We also denote by 
\[
C_{\infty}^{M}(I)=\bigcup_{0<\mu<\infty}C_{\mu}^{M}(I),\quad\text{and\ensuremath{\quad}}C_{0}^{M}(I)=\bigcap_{0<\mu<\infty}C_{\mu}^{M}(I),
\]
the \emph{Carleman and Beurling classes} respectively. Given $a\in(-\infty,\infty)$,
we denote the set of all \emph{right germs} and \emph{germs} of the
class $C_{\mu}^{M}$ at the point $a$ by $C_{\mu}^{M}(a_{+})$ and
$C_{\mu}^{M}(a)$ respectively, i.e.,
\[
C_{\mu}^{M}(a_{+}):=\bigcap_{\delta>0}C_{\mu}^{M}([a,a+\delta]),\quad\text{and}\quad C_{\mu}^{M}(a):=\bigcap_{\delta>0}C_{\mu}^{M}([a-\delta,a+\delta]).
\]
Finally, given a perfect compact set $\mathbb{E}\subset\mathbb{C}$
the \emph{analytic-Carleman} class $C_{\infty}^{M}(\mathbb{E})$ consists
of all functions $f\in\text{Hol\ensuremath{\left(\mathbb{E}^{\circ}\right)\cap C^{\infty}\left(\mathbb{E}\right)}}$
so that
\[
\left|f^{(n)}(z)\right|\leq C_{f}^{n+1}M_{n},\quad\forall z\in\mathbb{E},\,n\in\mathbb{Z}_{\geq0}
\]

From here on, all the sequences, $M$, will be assumed to be regular
in the following sense.
\begin{defn}
We say that a sequence of positive numbers, $M=(n!m_{n})_{n\geq0},$
is \emph{regular} if the following hold:
\end{defn}

\begin{enumerate}
\item $\left(m_{n}\right)$ is an eventually log-convex sequence (i.e., $m_{n}^{2}\leq m_{n-1}m_{n+1}$
for all sufficiently large $n$).
\item $\left(m_{n}^{1/n}\right)$ is eventually non-decreasing.
\item There exists a $C>0$ such that $m_{n+1}\leq Cm_{n}^{1+1/n}$ for any
$n\geq0$
\end{enumerate}
Put $\tau(M)=\lim m_{n}^{1/n}$ and note that if $\tau(M)<\infty$,
then $C^{M}(I)$ contains only analytic functions, while if $\tau(M)=\infty$,
then $C^{M}(I)$ also contains non-analytic functions. Therefore we
will refer to sequences $M$ with $\tau(M)<\infty$ as \emph{analytic
}and with $\tau(M)=\infty$ as \emph{non-analytic.}

The following criterion for quasianalyticity is due to Denjoy and
Carleman.
\begin{thm*}[Denjoy--Carleman \cite{carleman1926fonctions}]
Suppose that $M=\left(M_{n}\right)_{n\geq0}$ is a positive and log-convex
sequence. The class $C_{\mu}^{M}(I)$ (for $0\leq\mu\leq\infty$) is quasianalytic
if and only if 
\[
\sum_{n\geq0}\frac{M_{n}}{M_{n+1}}=+\infty.
\]
\end{thm*}

\subsubsection{Weighted Carleman classes: the class $A\left(M,\gamma,I\right)$\label{subsec:The-class A^M}}

As we already mentioned, we will consider $\mathcal{L}_{\gamma}$
on a more general domain of definition than $A$$\left(\gamma\right)$:
\begin{defn}
Let $\gamma$ be as before, $M$ be a positive sequence and let $J=\left[0,a\right]$
with $0<a<+\infty$. The class $A\left(M,\gamma,J\right)$ consists
of all functions $F\in C^{\infty}\left[0,\infty\right)$ such that
there exists a constant $C=C_{J}>0$ so that 
\[
\left|F^{(n)}\left(x\right)\right|\leq C^{n+1}M_{n}E\left(\frac{x}{a}\right).
\]
For an interval $J=\left[-a,0\right]$, we put 
\[
A\left(M,\gamma,J\right):=\left\{ F\in C^{\infty}\left(-\infty,0\right]\,:\,F(-x)\in A\left(M,\gamma,-J\right)\right\} .
\]
For a general interval $0\in I\subseteq\mathbb{R}$ (not necessarily
closed), we put 
\[
A\left(M,\gamma,I\right):=\bigcup_{J\subseteq I}A\left(M,\gamma,J\right),
\]
where the union is taken over all compact $J\subseteq I$ such that
$J\subseteq\left[0,\infty\right)$ or $J\subseteq\left(-\infty,0\right]$.
Finally, we put 
\[
A\left(M,\gamma,0_{+}\right):=\bigcup_{b>0}A\left(M,\gamma,\text{\ensuremath{\left[0,b\right]}}\right),\quad A\left(M,\gamma,0\right):=\bigcup_{b>0}A\left(M,\gamma,\text{\ensuremath{\left[-b,b\right]}}\right)
\]
\end{defn}

\begin{rem}
As we have already mentioned $A\left(\gamma\right)=A\left(M,\gamma,I\right)$
with $M_{n}=n!$ and $I=[0,1)$. Note that the elements of $A\left(M,\gamma,I\right)$
do not have to be analytic, but for $F\in A\left(M,\gamma,I\right)$
the function $\mathcal{L}_{\gamma}F$ is a well-defined $C^{\infty}\left(I\right)$
function.
\end{rem}

The next three classes we introduce will play an important role in
the description of $\mathcal{L}_{\gamma}A(M,\gamma,I)$.

\subsubsection{Carleman class under exponential change of variables: the class $B(M,N;I)$}

Given an interval $I$, put $\log I_{+}=\left\{ \log x:x\in I\cap[0,\infty)\right\} $
and $\log I_{-}=\left\{ \log\left(-x\right):x\in I\cap\left(-\infty,0\right]\right\} $

Let $M$ and $N$ be two regular sequences and let $I$\emph{ }be
an interval. The \emph{class} $B_{\mu}(M,N;I)$ consists of all functions
$f\in C^{\infty}(I)$ such that $f(x)\in C_{\infty}^{M}(I),$ $f(e^{x})\in C_{\mu}^{N}(\log I_{+})$
and $f\left(-e^{x}\right)\in C_{\mu}^{N}(\log I_{-})$. Similarly,
we define $B_{\eta}(M,N;0_{+})$ and $B_{\eta}(M,N;0)$ as the sets
of right-germs and germs at the origin respectively.

\subsubsection{Non-homogeneous Carleman classes}

The non-homogeneous Carleman class $C_{\mu}(M,N;I)$ consists of all
functions $f\in C^{\infty}(I)$, such that there exists a $C=C(f)>0$
such that 
\[
\left|f^{(n)}(x)\right|\leq C\cdot\min\left\{ C^{n}M_{n}\,,\,\frac{\mu^{n}\cdot N_{n}}{\left|x\right|^{n}}\right\} ,\quad\forall x\in I,n\in\mathbb{Z}_{\geq0}.
\]
If $I=[0,a]$, then non-homogeneous Carleman classes are closely related
to $B_{\mu}(M,N;I)$ as the next lemma demonstrates.
\begin{lem}
\label{lem:non-homg and B}Let $I=[0,a]$, and let $M$ and $N$ be
two regular sequences with $\tau\left(N\right)\geq\frac{2}{\pi}$. 
Then for any $\eta>1$
\[
B_{1/\eta}(M,N;I)\subseteq C_{1}(M,N^{*};I)\subseteq B_{\eta}(M,N;I).
\]
Here $N^{*}=N$ in the non-analytic case, and $N_{n}^{*}=n!\left(\sin\frac{1}{\tau(N)}\right)^{-n}$
in the analytic case.
\end{lem}

\subsubsection{Carleman type classes of analytic functions: the class $D\left(M,\gamma;I\right)$. }

Given an admissible weight $\gamma$ and $\eta>0$, put

\[
\Omega_{\eta}\left(\gamma\right):=\left\{ z\neq0:\sup_{t>0}E(t\eta)\left|K\left(\frac{t}{z}\right)\right|<\infty\right\} 
\]
and 
\[
D\left(M,\gamma;\left[0,a\right)\right)=\bigcup_{\eta>a}C_{\infty}^{M}(\overline{\Omega}_{\eta}\mathbb{\left(\gamma\right)}).
\]
That is, the class $D\left(M,\gamma;I\right)$ consists of all
functions $f$ such that for any $\eta>a$, $f$ is holomorphic in $\Omega_{\eta}\left(\gamma\right)$
and there exists a constant $C=C_{\eta}>0$ such that
\[
\left|f^{\left(n\right)}\left(z\right)\right|\leq C^{n+1}M_{n}\gamma_{n+1},\quad\forall n\in\mathbb{Z}_{+},\,\forall z\in\overline{\Omega}_{\eta}\left(\gamma\right).
\]
Equivalently, instead of using iterated derivatives, we can define
the class $D\left(M,\gamma;I\right)$ through the remainder of
the asymptotic expansion at the origin: given $f\in\text{Hol\ensuremath{\left(\Omega_{\eta}\left(\gamma\right)\right)}\ensuremath{\ensuremath{\cap C^{\infty}\left(\left\{ 0\right\} \cup\Omega_{\eta}\left(\gamma\right)\right)}}}$,
we put
\[
R_{n}(z,f):=f(z)-\sum_{0\leq k<n}\frac{f^{(k)}(0)}{k!}z^{n}.
\]
The class $D\left(M,\gamma;I\right)$ consists of all functions $f$
such that for any $\eta>a$, $f\in\text{Hol}\left(\Omega_{\eta}\left(\gamma\right)\right)\cap C^{\infty}\left(\left\{ 0\right\} \cup\Omega_{\eta}\left(\gamma\right)\right)$
and there exists a constant $C=C_{\eta}>0$ such that
\[
\left|R_{n}(z,f)\right|\leq C^{n+1}\frac{M_{n}\gamma_{n+1}}{n!}|z|^{n},\quad\forall n\in\mathbb{Z}_{+},\,\forall z\in\Omega_{\eta}\left(\gamma\right).
\]

\begin{rem}
Recall that $L(s)=\gamma(s)^{1/s},\,\varepsilon(s)=s\frac{L^{\prime}(s)}{L(s)}$
and 
\[
\widehat{\gamma}_{n}=\sup_{\rho>0}\rho^{n}\left|\gamma\left(i\rho\right)\right|=\sup_{\rho>0}\rho^{n}\exp\left(-\left(1+o(1)\right)\frac{\pi}{2}\rho\varepsilon(\rho)\right).
\]
If the sequence $\widehat{\gamma}$ is analytic (i.e., if $\lim_{\rho\to\infty}\varepsilon(\rho)>0$),
then $\Omega_{\eta}\left(\gamma\right)$ is indeed a domain, while
if the sequence $\widehat{\gamma}$ is non-analytic (i.e., if $\lim_{\rho\to\infty}\varepsilon(\rho)=0$),
then $[0,\eta^{-1}]\subseteq\Omega_{\eta}\left(\gamma\right)\subseteq[0,\eta^{-1}).$
In the latter case, $D\left(M,\gamma;I\right)=C_{\infty}^{M\cdot\gamma}(I).$
Therefore, we will consider $D\left(M,\gamma;I\right)$ only in
the analytic case.
\end{rem}

\section{Non-homogeneous Carleman classes}

\subsection{Equivalence of Non-homogeneous Carleman classes and the classes $B(M,N;I)$}

Throughout this section, we fix a regular sequence $N$. Recall that
$\tau(N)=\lim_{n\to\infty}\left(\frac{N_{n}}{n!}\right)^{1/n}$ and
put $N^{*}=N$ in the non-analytic case (i.e., $\tau(N)=\infty)$ and
$N_{n}^{*}=n!\frac{1}{\sin^{n}\left(1/\tau(N)\right)}$.

Lemma \ref{lem:non-homg and B} follows from the next lemma.
\begin{lem}
\label{lem:exp change of var in Carleman class} Let $f$ and $g$
be two $C^{\infty}$ functions related by $g(x)=f(e^{x})$ for $x\in\left[-\infty,0\right]$.
The function $f$ satisfies 
\[
\left|t^{n}f^{(n)}(t)\right|\lesssim_{\eta}\eta^{n}N_{n}^{*},\quad\forall t\in[0,1],\,n\in\mathbb{N},\,\eta>1,
\]

if and only if, $g\in C_{\eta}^{N}\left[-\infty,0\right]$ for all
$\eta>1.$
\end{lem}

\noindent To prove the Lemma, we will need the following combinatorial
claim.
\begin{claim}
\label{clm: striling number identiniy}Consider the Stirling numbers
of the first and second kind defined recursively by
\[
\begin{bmatrix}n+1\\
j
\end{bmatrix}=n\begin{bmatrix}n\\
j
\end{bmatrix}+\begin{bmatrix}n\\
j-1
\end{bmatrix},\quad\begin{bmatrix}n\\
0
\end{bmatrix}=0,\,\begin{bmatrix}n\\
n
\end{bmatrix}=1,
\]
and
\[
\begin{Bmatrix}n+1\\
j
\end{Bmatrix}=j\begin{Bmatrix}n\\
j
\end{Bmatrix}+\begin{Bmatrix}n\\
j-1
\end{Bmatrix},\quad\begin{Bmatrix}n\\
0
\end{Bmatrix}=0,\begin{Bmatrix}n\\
1
\end{Bmatrix}=\begin{Bmatrix}n\\
n
\end{Bmatrix}=1
\]
respectively. If $g(x)=f(e^{x})$ and $g\in C^{\infty},$ then
\[
g^{(n)}(x)=\sum_{j=1}^{n}\begin{Bmatrix}n\\
j
\end{Bmatrix}f^{(j)}\left(e^{x}\right)e^{jx},
\]
and
\[
f^{(n)}(t)=\frac{1}{t^{n}}\sum_{j=1}^{n}\left(-1\right)^{n+j}\begin{bmatrix}n\\
j
\end{bmatrix}g^{(j)}\left(\log t\right).
\]
\end{claim}

\begin{proof}
We will prove both of the identities by induction on $n$. Clearly,
the claim holds for $n=0$ and $n=1$. Assume that it is true for $j=n-1,$
i.e.,
\[
g^{(n-1)}(x)=\sum_{j=1}^{n-1}\begin{Bmatrix}n-1\\
j
\end{Bmatrix}f^{(j)}\left(e^{x}\right)e^{jx}.
\]
Differentiating this identity yields
\begin{align*}
g^{(n)}(x) & =\sum_{j=1}^{n-1}\begin{Bmatrix}n-1\\
j
\end{Bmatrix}\left(f^{(j+1)}\left(e^{x}\right)e^{\left(j+1\right)x}+jf^{(j)}\left(e^{x}\right)e^{jx}\right)\\
& =\sum_{j=2}^{n}\begin{Bmatrix}n-1\\
j-1
\end{Bmatrix}f^{(j)}\left(e^{x}\right)e^{jx}+\sum_{j=1}^{n-1}j\begin{Bmatrix}n-1\\
j
\end{Bmatrix}f^{(j)}\left(e^{x}\right)e^{jx}\\
& =\sum_{j=1}^{n}\left(\begin{Bmatrix}n-1\\
j-1
\end{Bmatrix}+j\begin{Bmatrix}n-1\\
j
\end{Bmatrix}\right)f^{(j)}\left(e^{x}\right)e^{jx}=\sum_{j=1}^{n}\begin{Bmatrix}n\\
j
\end{Bmatrix}f^{(j)}\left(e^{x}\right)e^{jx},
\end{align*}
proving the first identity. As for the second one, assume that
\[
f^{(n-1)}(t)=\frac{1}{t^{n-1}}\sum_{j=1}^{n-1}\left(-1\right)^{n-1+j}\begin{bmatrix}n-1\\
j
\end{bmatrix}g^{(j)}\left(\log t\right).
\]
Then
\begin{align*}
f^{(n)}(t) & =t^{-n}\sum_{j=1}^{n-1}\left(-1\right)^{n-1+j}\begin{bmatrix}n-1\\
j
\end{bmatrix}\left(g^{(j+1)}\left(\log t\right)+\left(1-n\right)g^{(j)}\left(\log t\right)\right)\\
& =t^{-n}\sum_{j=2}^{n}\left(-1\right)^{n+j}\begin{bmatrix}n-1\\
j-1
\end{bmatrix}g^{(j)}\left(\log t\right)+t^{-n}\sum_{j=1}^{n-1}\left(-1\right)^{n+j}\begin{bmatrix}n-1\\
j
\end{bmatrix}\left(n-1\right)g^{(j)}\left(\log t\right)\\
& =t^{-n}\sum_{j=1}^{n}\left(-1\right)^{n+j}\left(\begin{bmatrix}n-1\\
j-1
\end{bmatrix}+\begin{bmatrix}n-1\\
j
\end{bmatrix}\left(n-1\right)\right)g^{(j)}\left(\log t\right)\\
& = \frac{1}{t^{n}}\sum_{j=1}^{n}\left(-1\right)^{n+j}\begin{bmatrix}n\\
j
\end{bmatrix}g^{(j)}\left(\log t\right),
\end{align*}
proving the second identity.
\end{proof}
\begin{proof}[Proof of Lemma \ref{lem:exp change of var in Carleman class}]
We begin with the case $\tau(N)=+\infty$, so that $N^{*}=N.$ Assume
that $\left|t^{n}f^{(n)}(t)\right|\lesssim_{\eta}\eta^{n}N_{n}$ for
all $t\in[0,1],\,n\in\mathbb{N}$ and $\eta>1$. Then, $g\in C^{\infty}\left[-\infty,0\right],$
and satisfies the equality 
\[
g^{(n)}(x)=\sum_{j=1}^{n}\begin{Bmatrix}n\\
j
\end{Bmatrix}f^{(j)}\left(e^{x}\right)e^{jx},
\]
In particular, for any $\eta>1$
\[
\left|g^{(n)}(x)\right|\lesssim_{\eta}\eta^{n}\sum_{j=1}^{n}\begin{Bmatrix}n\\
j
\end{Bmatrix}N_{j}:=C\eta^{n}\widetilde{N}_{n}.
\]
We are going to prove by induction that $\widetilde{N}_{n}\leq C_{\delta}\left(1+\delta\right)^{n}N_{n}$,
for any $\delta>0.$ Indeed, assuming that the statement is valid
for $n$ and making use of the recursion formula, we get
\begin{align*}
\widetilde{N}_{n+1} & =\sum_{j=1}^{n+1}\begin{Bmatrix}n+1\\
j
\end{Bmatrix}N_{j}=\sum_{j=1}^{n+1}\left(j\begin{Bmatrix}n\\
j
\end{Bmatrix}+\begin{Bmatrix}n\\
j-1
\end{Bmatrix}\right)N_{j}\\
& =\sum_{j=1}^{n}j\begin{Bmatrix}n\\
j
\end{Bmatrix}N_{j}+\sum_{j=1}^{n}\begin{Bmatrix}n\\
j
\end{Bmatrix}N_{j+1}\leq\left(n+\frac{N_{n+1}}{N_{n}}\right)\widetilde{N}_{n}\\
& \leq C_{\delta}\left(1+\delta\right)^{n}\left(n+\frac{N_{n+1}}{N_{n}}\right)N_{n}.
\end{align*}
Since $\tau(N)=+\infty$, for sufficiently large $n$ we have $\left(n+\frac{N_{n+1}}{N_{n}}\right)\leq\left(1+\delta\right)\frac{N_{n+1}}{N_{n}}$, and
the statement is proved. We conclude that $g\in C_{\eta}^{N}\left[-\infty,0\right]$
for all $\eta>1.$

Next, we assume that \textbf{$g\in C_{\eta}^{N}\left[-\infty,0\right]$}.
Since, $f(t)=g\left(\log t\right)$, we obtain
\[
f^{(n)}(t)=\frac{1}{t^{n}}\sum_{j=1}^{n}\left(-1\right)^{n+j}\begin{bmatrix}n\\
j
\end{bmatrix}g^{(j)}\left(\log t\right),
\]
The same induction argument as before yields that $\left|t^{n}f^{(n)}(t)\right|\lesssim_{\eta}\eta^{n}N_{n}.$

Finally, we treat the case $\tau(N)<+\infty.$ In this case, the function
$g$ is analytic in the half-strip
\[
\left\{ z:\text{dist\ensuremath{\left(z,\mathbb{R}_{-}\right)}<\ensuremath{\frac{1}{\tau(N)}}}\right\} ,
\]
while $f$ is analytic in the sector 
\[
\left\{ s:\left|\arg s\right|<\arcsin\frac{1}{\tau(N^{*})}\right\} .
\]
Since the map $s=e^{z}$ maps the above half-strip into the sector in question, the lemma is proved.
\end{proof}

\subsection{Quasianalyticity~of non--homogeneous Carleman classes\label{subsec:Quasianalyticity}}

Here we discuss quasianalyticity in the non--homogeneous Carleman
classes $B(M,N;I).$ We begin our discussion by recalling some of
the classical results about quasianalyticity in $C_{\infty}^{M}\left(\Omega\right)$
where $\Omega$ is an interval that contains the origin, $I,$ or
a sector of opening $\pi\alpha,$ $\Omega_{\alpha}=\left\{ z:\left|\arg z\right|\leq\alpha\frac{\pi}{2},0\leq|z|\leq\rho_{0}\right\} .$

As we already mentioned, Denjoy and Carleman \cite{carleman1926fonctions}
found a necessary and sufficient condition for the Carleman class
$C^{M}(I)$ to be quasianalytic. For a regular sequence $M,$ their
condition is

\[
\sum_{n\geq0}\frac{M_{n}}{M_{n+1}}=\infty.
\]

A result proved by Carleson \cite{carleson1952sets}, Salinas \cite{salinas1955funciones}
and Korenblum \cite{korenblum1965quasianalytic,korenblum1966non}
implies that the condition 
\[
\sum_{n\geq0}\left(\frac{M_{n}}{M_{n+1}}\right)^{1+\frac{1}{\alpha}}=\infty
\]
is necessary and sufficient for quasianalyticity of the class $C_{\infty}^{M}\left(\Omega_{\alpha}\right)$
provided that $\alpha<2$. These classical results provide us with
quasianalyticity criterion for non--homogeneous Carleman classes,
$B(M,N;I)$, in the case that $B(M,N;I)$ is reduced to a Carleman
class (i.e. when $N_{n}^{1/n}\lesssim M_{n}^{1/n})$ or in the case
that $N$ is analytic (i.e. when $\left(N_{n}/n!\right)^{1/n}\lesssim1),$
but, in general, quasianalyticity of non--homogeneous Carleman classes
is not covered by the classical theory of quasianalyticity. Here,
we are presenting a criterion for quasianalyticity which follows from
the considerations of the image of $\mathcal{B}_{\gamma}$ under different
classes.
\begin{thm}
\label{thm: quasianalytic}Suppose that $\gamma$ is an admissible
weight with non-analytic $\widehat{\gamma}$, $M$ is a regular
sequence, $\eta>1$ and that $I$ is an interval containing the origin.
Then, the following hold:
\begin{enumerate}
\item If $\sum_{n\geq1}\frac{M_{n}}{M_{n+1}}=\infty$, then the classes
$B_{1/\eta}(M\gamma,\widehat{\gamma};I)$ and $C_{1/\eta}(M\gamma,\widehat{\gamma};I)$
are quasianalytic.
\item If $\sum_{n\geq1}\frac{M_{n}}{M_{n+1}}<\infty$, then the classes
$B_{\eta}(M\gamma,\widehat{\gamma};I)$ and $C_{\eta}(M\gamma,\widehat{\gamma};I)$
are not quasianalytic.
\end{enumerate}
\end{thm}

\begin{rem}
Note that $B_{\eta}(M\gamma,\widehat{\gamma};I)$ is always quasianalytic
away from zero (i.e., the map $f\mapsto\left(f^{(n)}(c)\right)_{n\geq0}$
is injective for any $c\in J$, $c\neq0$). Indeed, if $f\in B_{\eta}(M\gamma,\widehat{\gamma};I)$,
then $f\vert_{I\setminus\left\{ 0\right\} }\in C_{\infty}^{\widehat{\gamma}}\left(I\setminus\left\{ 0\right\} \right)$.
The first assumption of admissible functions implies that $\sum\widehat{\gamma}_{n}^{-1/n}=\infty$
and thus, by the Denjoy--Carleman Theorem, the class $C^{\widehat{\gamma}}\left(J\setminus\left\{ 0\right\} \right)$
is quasianalytic.
\end{rem}

\begin{rem}
The second assertion of the above Theorem follows immediately from
Theorem \ref{thm:Main thm}.
\end{rem}

\begin{rem}
It is sufficient to prove the first assertion for germs. Indeed, by
reflecting $I$ if necessary, there is no loss of generality with
assuming $[0,\delta]\subset I$ for sufficiently small $\delta$.
Thus, if $f\in B_{1/\eta}(M\gamma,\widehat{\gamma};I)$, with $f^{(n)}(0)\equiv0$
for all $n\geq0,$ $f\vert_{[0,\delta]}$ can be extended to an element
of $B_{1/\eta}(M\gamma,\widehat{\gamma};[-\delta,\delta])$.
\end{rem}

\begin{example}
The class $B_{1}(n!\log^{n\alpha}\left(n+e\right),n!\log^{n}\left(n+e\right);I)$
is \emph{quasianalytic} if $\alpha<\frac{\pi}{2}+1$ and is \emph{not
quasianalytic} if $\alpha>\frac{\pi}{2}+1$. The theorem doesn't treat
the case $\alpha=\frac{\pi}{2}+1$, but we suspect that in this case
the class is also quasianalytic.
\end{example}

\begin{example}
The class $B_{1}\left(n!\exp\left(\alpha n\sqrt{\log(n+1)}\right),n!\log^{n/2}\left(n+e\right);I\right)$
is \emph{quasianalytic} if $\alpha\leq\pi$ and is \emph{not quasianalytic}
if $\alpha>\pi$. 
\end{example}

\begin{example}
The class $B_{1}(n!\log^{n}\left(n+e\right)\log^{\beta n}\log(n+e^{e}),n!\log^{n}\left(n+e\right)\log^{n}\log(n+e^{e});I)$
is \emph{quasianalytic} if $\beta<\frac{\pi}{2}+1$ and is \emph{not
quasianalytic} if $\beta>\frac{\pi}{2}+1$.
\end{example}

\begin{rem}
A stronger quasianalyticity result with an elementary proof can be
obtained by Hirschman's version of the uncertainty principle \cite{hirschman1950behaviour}.
The statement and its proof will be presented elsewhere.
\end{rem}

\begin{rem}
The result of Carleson, Salinas and Korenblum remains true for $\alpha\geq2$,
provided that we consider $\Omega_{\alpha}$ as a sector in the Riemann
surface of the logarithm, i.e., if $C_{\infty}^{M}\left(\Omega_{\alpha}\right)$
consists of functions $f$ so that $f(e^{z})$ is analytic in the
half strip
\[
\left\{ w:\text{Re}\left(w\right)<\log\rho_{0},\,\left|\text{Im}\left(w\right)\right|<\alpha\cdot\frac{\pi}{2}\right\} .
\]
\end{rem}

\subsubsection{Factorization of functions in  $B_{\eta}(M,N;0)$}

The proof of Theorem \ref{thm: quasianalytic} requires a factorization
of functions in $B_{\eta}(M,N;0)$ to a sum of two functions which
are analytic in the lower and upper half-plane respectively. Put $S_{\mu}^{\pm}=\left\{ z:\text{\ensuremath{\pm}Im}z>0,\,|z|<\mu\right\} .$
The \emph{class} $B_{\eta}^{\pm}(M,N;0)$ is the class of all functions
$f\in B_{\eta}(M,N;0)$ for which there exists $\mu=\mu_{f}>0$ such
that $f\in C^{\infty}\left(\overline{S_{\mu}^{\pm}}\right)\cap\text{Hol}\left(S_{\mu}^{\pm}\right).$

Let $\gamma$ be an admissible weight with non-analytic $\widehat{\gamma}$
and let $M$ be a quasianalytic regular sequence (i.e., $\sum_{n\geq1}\frac{M_{n}}{M_{n+1}}=\infty$).
In Section \ref{subsec:The-classes plus minus} we will prove the
factorization:
\begin{equation}
B_{\eta}(M\gamma,\widehat{\gamma};0)\subseteq B_{\eta_{1}}^{+}(M\gamma,\widehat{\gamma};0)+B_{\eta_{1}}^{-}(M\gamma,\widehat{\gamma};0),\quad\forall\eta_{1}>\eta>0.\label{eq:Factorization =00005Cpm}
\end{equation}
Note that the above factorization is not unique, but $B_{\eta}^{+}(M\gamma,\widehat{\gamma};0)\cap B_{\eta}^{-}(M\gamma,\widehat{\gamma};0)=C^{\omega}\left(0\right)$, which
is the class of all real analytic germs at the origin

Also observe that 
\[
\sum_{n\geq1}\frac{M_{n}}{M_{n+1}}=\infty\quad\Rightarrow\quad\sum_{n\geq1}\sqrt{\frac{M_{n}\gamma_{n}}{M_{n+1}\gamma_{n+1}}}=\infty,
\]
and therefore by the result of Carleson, Salinas and Korenblum,
the classes $B_{\eta}^{\pm}(M\gamma,\widehat{\gamma};J)$ are quasianalytic.
The next lemma, which will be proven in Section \ref{subsec:The-classes plus minus},
is the key ingredient in the proof of Theorem \ref{thm: quasianalytic}.
\begin{lem}
\label{lem:intersection pm}For any $\eta>0$, 
\[
\mathcal{B}_{\gamma}B_{\eta}^{+}(M\gamma,\widehat{\gamma};0)\cap\mathcal{B}_{\gamma}B_{\eta}^{-}(M\gamma,\widehat{\gamma};0)=\mathcal{B}_{\gamma}C^{\omega}\left(0\right),
\]

where $C^{\omega}\left(0\right)$ is the class of real analytic germs
at the origin. 
\end{lem}

\begin{proof}[Proof of Theorem \ref{thm: quasianalytic}]
Let $\eta>1$ and $f\in  B_{1/\eta}(M\gamma,\widehat{\gamma};J)$
with $f^{(n)}(0)\equiv0$ for all $n\geq0.$ Consider the factorization
$f=f_{+}+f_{-}$ with $f_{\pm}\in B_{\eta_{0}}(M\gamma,\widehat{\gamma};J)$
for some $1/\eta<\eta_{0}<1$. By the linearity of $\mathcal{B_{\gamma}},$
we have $\mathcal{B_{\gamma}}f=\mathcal{B_{\gamma}}f_{+}+\mathcal{B_{\gamma}}f_{-}$.
By Theorem \ref{thm:Main thm non analytic}, $\mathcal{B_{\gamma}}f\in A(M,\gamma;0)$
with $\left(\mathcal{B_{\gamma}}f\right)^{(n)}(0)\equiv0$ for all
$n\geq0$. Since $A(M,\gamma;0)\subset C^{M}\left(0\right)$ and the latter is quasianalytic, we conclude that 
\[
\mathcal{B_{\gamma}}f\equiv0\quad\Rightarrow\quad\mathcal{B_{\gamma}}f_{+}\equiv-\mathcal{B_{\gamma}}f_{-}.
\]
Therefore $\mathcal{B_{\gamma}}f_{\pm}\in\mathcal{B}_{\gamma}B_{\eta}^{+}(M\gamma,\widehat{\gamma};0)\cap\mathcal{B}_{\gamma}B_{\eta}^{-}(M\gamma,\widehat{\gamma};0)$,
which by Lemma \ref{lem:intersection pm} implies that $\mathcal{B_{\gamma}}f_{\pm}\in\mathcal{B}_{\gamma}\left(C^{\omega}\left(0\right)\right)$.
Since the map $\mathcal{B}_{\gamma}$ is injective in the classes
$B_{\eta}^{\pm}(M\gamma,\widehat{\gamma};J)$ and $C^{\omega}\left(0\right),$ we
have that $f_{\pm}\in C^{\omega}\left(0\right).$ We conclude that
$f\in C^{\omega}\left(0\right)$, with $f^{(n)}(0)\equiv0$ for all
$n\geq0,$ and therefore $f\equiv0$ as claimed.
\end{proof}
\begin{rem}
The same technique was used in \cite{kiro2018taylor}, to deduce a
new Phragm\'{e}n--Lindel\"of type theorem from quasianalytic Beurling
classes $C_{0}^{M}\left(I\right).$
\end{rem}

\section{The inclusion $\lpg A(M,\gamma;I)\subseteq B_{\eta}(M\gamma,\widehat{\gamma};I)$}

\subsection{Estimates for the derivative of $K$}

Recall that 
\[
K(t)=\frac{1}{2\pi i}\int_{c-i\infty}^{c+i\infty}t^{-z}\gamma(z)dz,\quad c>0,
\]
and put 
\[
K_{1}\left(t\right):=K\left(e^{t}\right)=\frac{1}{2\pi i}\int_{c-i\infty}^{c+i\infty}e^{-zt}\gamma(z)dz.
\]

\begin{lem}
\label{lem: K_1 deriv estimate}If $\gamma$ is admissible with non-analytic $\widehat{\gamma}$ \emph{(}i.e. with $\lim_{\rho\to\infty}\varepsilon\left(\rho\right)=0$\emph{)},
then for any $\delta>0$, we have
\[
\left|\frac{d^{n}}{dt^{n}}\left(e^{t}K_{1}\left(t\right)\right)\right|\lesssim_{\delta}(1+\delta)^{n}\widehat{\gamma}_{n}K_{1}(t-\delta),
\]
\end{lem}

\begin{proof}
First we will show
\begin{equation}
\left|K_{1}^{\left(n\right)}\left(t\right)\right|\lesssim_{\delta}(1+\delta)^{n}\widehat{\gamma}_{n}K_{1}(t-\delta).\label{eq: K_1 estimate}
\end{equation}
We write
\[
K_{1}^{\left(n\right)}\left(t\right)=\frac{\left(-1\right)^{n}}{2\pi i}\int_{\Gamma}z^{n}e^{-zt}\gamma(z)dz,
\]
for a curve $\Gamma$ that we will specify below. Fix $\delta>0$
and denote by $\rho_{0}=\rho_{0}\left(t,\delta\right)$ the solution
to 
\[
\log L\left(\rho\right)+\varepsilon\left(\rho\right)=t-\delta.
\]
Assume that $t$ is sufficiently large so that $\varepsilon\left(\rho_{0}\right)<\delta$,
and consider the curve 
\[
\Gamma=\Gamma_{1}+\Gamma_{2}+\Gamma_{3}=\left(-i\infty,-i\rho_{0}\right]+\left\{ \rho_{0}e^{i\theta}:\theta\in\left[-\tfrac{\pi}{2},\tfrac{\pi}{2}\right]\right\} +\left[i\rho_{0},i\infty\right).
\]
\begin{figure}[h]

\includegraphics[scale=0.6]{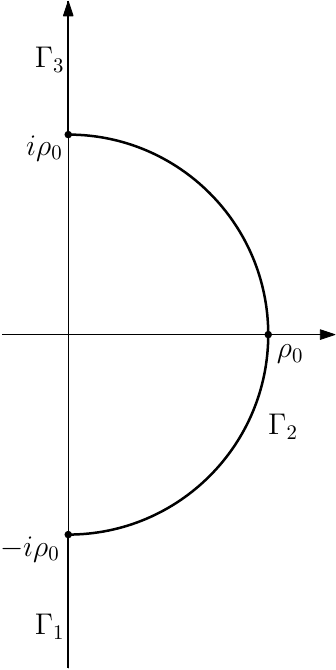}\caption{$\Gamma=\Gamma_{1}+\Gamma_{2}+\Gamma_{3}$}
\end{figure}

By (\ref{eq:gamma asymp full}) we have for $z=\rho_{0}e^{i\theta}$
\begin{align*}
\left|e^{-zt}\gamma(z)\right| & =\exp\left[\rho_{0}\cos\theta\left(\log L\left(\rho_{0}\right)-t\right)-\rho_{0}\theta\sin\theta\varepsilon\left(\rho_{0}\right)+o\left(\rho_{0}\varepsilon\left(\rho_{0}\right)\right)\right]\\
& =\exp\left[-\rho_{0}\cos\theta\left(\varepsilon\left(\rho_{0}\right)+\delta\right)-\rho_{0}\theta\sin\theta\varepsilon\left(\rho_{0}\right)+o\left(\rho_{0}\varepsilon\left(\rho_{0}\right)\right)\right]\\
& \leq\exp\left[-2\cos\theta\rho_{0}\varepsilon\left(\rho_{0}\right)-\rho_{0}\theta\sin\theta\varepsilon\left(\rho_{0}\right)+o\left(\rho_{0}\varepsilon\left(\rho_{0}\right)\right)\right]
\end{align*}
Notice that in the interval $\left[0,\tfrac{\pi}{2}\right],$ the
function $\theta\mapsto-2\cos\theta-\theta\sin\theta$ is increasing
and bounded from above by $-\frac{\pi}{2}.$ So, for $z\in\Gamma_{2}$
and $t$ sufficiently large, we have 
\[
\left|e^{-zt}\gamma(z)\right|\leq\exp\left[-\tfrac{\pi}{2}\rho_{0}\varepsilon\left(\rho_{0}\right)\left(1+o\left(1\right)\right)\right]\leq\exp\left[-\left(\tfrac{\pi}{2}-2\delta\right)\rho_{0}\varepsilon\left(\rho_{0}\right)+\delta\rho_{0}\varepsilon\left(\rho_{0}\right)\right]
\]
Thus, by Theorem A, we have 
\[
e^{\delta\rho_{0}\varepsilon\left(\rho_{0}\right)}\lesssim_{\delta}\left(K_{1}\left(t-2\delta\right)\right)^{\delta}\ensuremath{\lesssim_{\delta}}K_{1}\left(t-3\delta\right).
\]
We conclude that 
\begin{align*}
\left|\int_{\Gamma_{2}}z^{n}e^{-zt}\gamma(z)dz\right| & \lesssim_{\delta}K_{1}\left(t-3\delta\right)\rho_{0}^{n+1}e^{-\left(\tfrac{\pi}{2}-2\delta\right)\rho_{0}\varepsilon\left(\rho_{0}\right)}\\
& \leq\rho_{0}K_{1}\left(t-3\delta\right)\sup_{\rho>0}\rho^{n}e^{-\left(\tfrac{\pi}{2}-2\delta\right)\rho\varepsilon\left(\rho\right)}.
\end{align*}
Since the function $\varepsilon$ is slowly varying, we get 
\[
\sup_{\rho>0}\rho^{n}e^{-\left(\tfrac{\pi}{2}-2\delta\right)\rho\varepsilon\left(\rho\right)}\lesssim_{\delta}\left(1+3\delta\right)^{n}\sup_{\rho>0}\rho^{n}e^{-\tfrac{\pi}{2}\rho\varepsilon\left(\rho\right)}\lesssim_{\delta}\left(1+4\delta\right)^{n}\widehat{\gamma}_{n}.
\]
We obtained
\[
\left|\int_{\Gamma_{2}}z^{n}e^{-zt}\gamma(z)dz\right|\lesssim_{\delta}\left(1+4\delta\right)^{n}\widehat{\gamma}_{n}K_{1}\left(t-4\delta\right).
\]
If $z=i\rho\in\Gamma_{3},$ then 
\[
\left|e^{-zt}\gamma(z)\right|=\exp\left[-\tfrac{\pi}{2}\rho\varepsilon\left(\rho\right)\left(1+o\left(1\right)\right)\right].
\]
Thus
\begin{align*}
\left|\int_{\Gamma_{3}}z^{n}e^{-zt}\gamma(z)dz\right| & \lesssim_{\delta}e^{-\delta\rho_{0}\varepsilon\left(\rho_{0}\right)}\cdot\left(\sup_{\rho>0}\rho^{n}e^{-\left(\tfrac{\pi}{2}-2\delta\right)\rho\varepsilon\left(\rho\right)}\right)\int_{\rho_{0}}^{\infty}e^{-\delta\rho\varepsilon\left(\rho\right)}\\
& \lesssim_{\delta}K_{1}\left(t-4\delta\right)\cdot\left(1+4\delta\right)^{n}\widehat{\gamma}_{n}.
\end{align*}
The bound for $\int_{\Gamma_{1}}$ is obtained similarly, proving
$\eqref{eq: K_1 estimate}$. To finish the proof, we observe that
\[
\frac{d^{n}}{dt^{n}}\left(e^{t}K_{1}\left(t\right)\right)=e^{t}\sum_{k=0}^{n}{n \choose k}K_{1}^{(k)}(t).
\]
So 
\[
\left|\frac{d^{n}}{dt^{n}}\left(e^{t}K_{1}\left(t\right)\right)\right|\lesssim_{\delta}e^{t}K_{1}(t-\delta)\sum_{k=0}^{n}{n \choose k}\left(1+\delta\right)^{k}\widehat{\gamma}_{k}.
\]
Since $\left(\widehat{\gamma}_{k}\right)_{k}$ is eventually log-convex
with $\left(\frac{\widehat{\gamma}_{k}}{k!}\right)^{1/k}\to\infty,$
we have 
\[
\left|\frac{d^{n}}{dt^{n}}\left(e^{t}K_{1}\left(t\right)\right)\right|\lesssim_{\delta}\left(1+\delta\right)^{n}e^{t}K_{1}(t-\delta).
\]
Finally, the proof is complete by noticing that
\[
e^{t}K_{1}(t-\delta)\lesssim_{\delta}K_{1}(t-2\delta).
\]
\end{proof}

\subsection{Proof of the inclusion $\mathcal{L}_{\gamma}A\left(M,\gamma,I\right)\subseteq B_{\eta}\left(M\gamma,\widehat{\gamma};I\right)$
in the non-analytic case\label{subsec:1st inc}}

The proof of the inclusion $R_{\gamma}A\left(M,\gamma,I\right)\subseteq B_{\eta}\left(M\gamma,\widehat{\gamma};I\right)$,
requires the following lemma, its proof is given in Appendix A. 
\begin{lem}
\label{lem:three E inq}For any $\eta<1$, there exists $\delta=\delta(\eta)>0$
so that 
\[
E\left(t\delta\right)E(t\eta)\lesssim_{\eta}E\left(t\right),\quad\text{and \ensuremath{\quad E\left(t\delta\right)E(t\eta)\lesssim_{\eta}\frac{1}{\left|K(t)\right|},}}\quad\forall t>0.
\]
\end{lem}

\begin{proof}[Proof of the inclusion $R_{\gamma}A\left(M,\gamma,I\right)\subseteq B_{\eta}\left(M\gamma,\widehat{\gamma};I\right)$
in the non-analytic case]
Without loss of generality, we may assume that $I=\left[0,1\right).$
Let $G\in A\left(M,\gamma,I\right)$ and $J=\left[0,a\right]\subset I$.
Put 
\[
f(x)=\left(\mathcal{L}_{\gamma}G\right)\left(x\right)=\int_{0}^{\infty}G\left(xt\right)K\left(t\right)dt.
\]
By definition, there exists a $C>0$ such that 
\[
\left|G^{(n)}(t)\right|\lesssim_{\eta}C^{n}M_{n}E\left(t\eta\right),
\]
for any $\eta>1.$ Choosing $\eta$ so close to $1$ such that
$a<\frac{1}{\eta^{2}}$, for any $x\in J$, we have 
\begin{align*}
\left|f^{(n)}(x)\right| & \leq\left|\int_{0}^{\infty}t^{n}G^{(n)}\left(xt\right)K\left(t\right)dt\right|\lesssim C^{n}M_{n}\int_{0}^{\infty}t^{n}E\left(tx\eta\right)\left|K(t)\right|dt\lesssim C^{n}\left(M_{n}\int_{1}^{\infty}t^{n}E\left(t\eta^{-1}\right)\left|K(t)\right|dt+1\right).
\end{align*}
By Lemma \ref{lem:three E inq}, there is $\delta>0$ such that $E(\delta t)E\left(t\eta^{-1}\right)\left|K(t)\right|\lesssim 1$.
Therefore 
\[
\left|f^{(n)}(x)\right|\lesssim C^{n}\left(M_{n}\int_{1}^{\infty}t^{n}\frac{1}{E\left(\delta t\right)}dt+1\right)\lesssim C^{n}\left(M_{n}\int_{1}^{\infty}t^{n}\frac{\gamma(n+3)}{\left(\delta t\right)^{n+2}}dt+1\right)\lesssim\left(\frac{C}{\delta}\right)^{n}M_{n}\gamma(n+3).
\]
Since $\left(\frac{\gamma(n+3)}{\gamma(n+1)}\right)^{1/n}\to 1$, we
conclude that $f\in C_{\infty}^{M\cdot\gamma}\left(I\right).$

Next we will show $f\left(e^{x}\right)\in C_{\eta}^{\widehat{\gamma}}\left(\left[-\infty,0\right)\right)$
for any $\eta>1.$ Put $K_{1}\left(u\right):=K\left(e^{u}\right)$
and $K_{2}(u)=e^{u}K_{1}(u)$. Then
\[
f\left(e^{x}\right)=\int_{0}^{\infty}G\left(e^{x}t\right)K\left(t\right)dt\stackrel{t=e^{u}}{=}\int_{\mathbb{R}}G\left(e^{x+u}\right)K_{1}(u)e^{u}du=\int_{\mathbb{R}}G\left(e^{y}\right)K_{1}(y-x)e^{y-x}dy=\int_{\mathbb{R}}G\left(e^{y}\right)K_{2}(y-x)dy.
\]
Thus
\[
\frac{d^{n}}{dx^{n}}f\left(e^{x}\right)=(-1)^{n}\int_{\mathbb{R}}G\left(e^{y}\right)K_{2}^{(n)}\left(y-x\right)dy
\]
By Lemma \ref{lem: K_1 deriv estimate}, we have 
\[
\left|K_{2}^{(n)}\left(y-x\right)\right|\lesssim_{\delta}(1+\delta)^{n}\widehat{\gamma}_{n}\left|K_{1}(y-x-\delta)\right|,
\]
so we conclude that 
\[
\left|\frac{d^{n}}{dx^{n}}f\left(e^{x}\right)\right|\lesssim_{\delta}(1+\delta)^{n}\widehat{\gamma}_{n}\int_{\mathbb{R}}G\left(e^{y}\right)\left|K_{1}\left(y-x-\delta\right)\right|dy.
\]
Let $J$ be a compact subset of $\left[-\infty,0\right).$ We choose
$\delta$ small enough so that $-x-\delta>\frac{\delta}{2}$, for any
$x\in J$. Thus 
\[
\left|\frac{d^{n}}{dx^{n}}f\left(e^{x}\right)\right|\lesssim_{\delta}(1+\delta)^{n}\widehat{\gamma}_{n}\int_{\mathbb{R}}E\left(e^{y+\delta/4}\right)\left|K_{1}\left(y+\tfrac{\delta}{2}\right)\right|dy\lesssim_{\delta}(1+\delta)^{n}\widehat{\gamma}_{n}.
\]
We obtained that $f\left(e^{x}\right)\in C_{\eta}^{\widehat{\gamma}}\left(\left[-\infty,0\right)\right),$
for any $\eta>1$, and hence $f\in B_{\eta}\left(M\gamma,\widehat{\gamma};I\right)$.
\end{proof}

\subsection{The inclusion $\lpg A\left(M,\gamma;I\right)\subseteq D\left(M,\gamma;I\right)$
in the analytic case}

Throughout this section, we assume that $\gamma$ is an admissible
weight with analytic $\widehat{\gamma}$ (i.e., with $\lim_{\rho\to\infty}\varepsilon(\rho)>0)$
and that $I=\left[0,1\right)$. Recall that 
\[
\Omega_{\eta}\left(\gamma\right):=\left\{ z:\sup_{t>0}E(t\eta)\left|K\left(\frac{t}{z}\right)\right|<\infty\right\} \text{.}
\]

\begin{proof}[Proof of the inclusion $\lpg A\left(M,\gamma;I\right)\subseteq D\left(M,\gamma;I\right)$]
. For $\eta>1,$ assume that a function $F$ satisfies
\[
|F^{(n)}(t)|\leq C^{n+1}M_{n}E(t\eta).
\]
Put $f=\lpg F$. By Lemma \ref{lem:three E inq}, there exists a $\delta>0$
so that for any $z\in\Omega_{\eta^{2}}(\gamma)$, we have 
\[
\left|\int_{0}^{\infty}F(t)K\left(\frac{t}{z}\right)\frac{dt}{z}\right|\lesssim\frac{1}{\left|z\right|}\int_{0}^{\infty}\frac{E(t\eta)}{E\left(t\eta^{2}\right)}dt\lesssim\frac{1}{\left|z\right|}\int_{0}^{\infty}\frac{1}{E\left(t\delta\right)}dt<\infty.
\]
So, by dominated the convergence theorem, the function 
\[
f(z)=\int_{0}^{\infty}F(t)K\left(\frac{t}{z}\right)\frac{dt}{z}
\]
is holomorphic in $\Omega_{\eta^{2}}(\gamma)$. For a given $n\geq0$,
write 
\[
F(t)=\sum_{0\leq k<n}\frac{F^{(n)}(0)}{n!}t^{n}+R_{n}(t,F),
\]
and notice that 
\[
R_{n}(f,z)=\int_{0}^{\infty}R_{n}(t,F)K\left(\frac{t}{z}\right)\frac{dt}{z}.
\]
Writing $z=re^{i\theta}$, and making the change of variable $t=r\cdot u$,
we have 
\[
R_{n}(f,z)=\int_{0}^{\infty}R_{n}(ru,F)K\left(ue^{-i\theta}\right)e^{-i\theta}du.
\]
By the Lagrange form of the remainder $R_{n}(t,F)$, we have for $t\geq0$
\[
\left|R_{n}(t,F)\right|\leq C^{n+2}\frac{M_{n+1}}{\left(n+1\right)!}E(t\eta)t^{n+1}.
\]
Therefore by Lemma \ref{lem:three E inq}, there is $\delta>0$ so
that 
\[
\left|R_{n}(f,z)\right|\leq C^{n+2}\frac{M_{n+1}}{\left(n+1\right)!}\int_{0}^{\infty}\left(ru\right)^{n+1}E(rt\eta)\left|K\left(ue^{-i\theta}\right)\right|du\lesssim_{\eta}C_{\eta}^{n}\frac{M_{n+1}}{\left(n+1\right)!}|z|^{n+1}\int_{0}^{\infty}\frac{t^{n}}{E\left(\delta t\right)}dt.
\]
By definition 
\[
E\left(\delta t\right)\geq\frac{\delta^{n+3}t^{n+2}}{\gamma_{n+3}},
\]
so, we have
\[
\int_{0}^{\infty}\frac{t^{n}}{E\left(\delta t\right)}dt=\int_{0}^{1}\frac{t^{n}}{E\left(\delta t\right)}dt+\int_{1}^{\infty}\frac{t^{n}}{E\left(\delta t\right)}dt\lesssim_{\eta}1+\frac{\gamma(n+3)}{\delta^{n}}\lesssim_{\eta}\frac{\gamma(n+3)}{\delta^{n}}.
\]
Since there is a $C>0$ such that 
\[
\frac{M_{n+1}}{M_{n}}\leq C^{n+1}\quad\text{and\ensuremath{\quad}}\frac{\gamma_{n+3}}{\gamma_{n+1}}\leq C^{n+1},
\]
we have established
\[
\left|R_{n}(f,z)\right|\leq C_{\eta}^{n}\frac{M_{n}}{n!}\gamma_{n+1}|z|^{n+1},\quad\forall z\in\Omega_{\eta^{2}}(\gamma).
\]
This establishes the estimate, as claimed.
\end{proof}

\section{Almost holomorphic extensions in Non-homogeneous Carleman classes}

\subsection{Almost holomorphic extensions of functions in $B(\widehat{\gamma},M;I)$:
the non-analytic case}

Throughout this section we fix an admissible weight $\gamma$ with
non-analytic $\widehat{\gamma}$ and a regular sequence $M=\left(n!m_{n}\right)_{n\geq0}$,
and put
\[
\gamma_{n}=L^{n}(n),\quad\widehat{\gamma}_{n}=n!\widehat{L}^{n}(n).
\]

\begin{claim}
\label{claim:m_k reg first claim}Let $C_{1}>0$. There exist constants
$C_{2},C_{3},C_{4}>0$ so that
\end{claim}

\begin{enumerate}
\item $1\le m_{k}^{1/k}\leq m_{1}\cdot k^{C_{2}},$ for any $k\geq2$;
\item $m_{k}^{1/k}\leq m_{n}^{1/n}\leq C_{3}m_{k}^{1/k}$ for any $k\in\mathbb{N}$
and any $n\in\left[k,C_{1}k\right]$;
\item $m_{k}^{n/k}\leq C_{4}^{k+n}m_{n}$ for any $n,k\in\mathbb{N}.$
\end{enumerate}
\begin{proof}
By regularity assumption 3, we have 
\begin{equation}
m_{k}^{1/k}\leq m_{k+1}^{1/(k+1)}\leq C^{1/\left(k+1\right)}m_{k}^{1/k}\label{eq:temp m_k^1/k reg 3}
\end{equation}
Therefore
\[
m_{k}^{1/k}\leq C^{\sum_{j=1}^{k}\frac{1}{j}}m_{0}\leq\left(2C\right)^{\log\left(k+1\right)}\leq k^{C_{2}}m_{1},
\]
proving the first assertion.

As for the second assertion, if $n\geq k$, then iterating \ref{eq:temp m_k^1/k reg 3}
yields 
\[
m_{k}^{1/k}\leq m_{n}^{1/n}\leq m_{k}^{1/k}\cdot\exp\left[\log C\sum_{j=k+1}^{n}\frac{1}{k}\right].
\]
In particular, if $n\leq C_{1}k$, then we get 
\[
m_{n}^{1/n}\leq m_{k}^{1/k}\cdot\exp\left[\log C\sum_{j=k+1}^{C_{1}k}\frac{1}{k}\right]\leq m_{k}^{1/k}\cdot\exp\left[\log C\left(\log C_{1}+1\right)\right]:=m_{k}^{1/k}C_{3},
\]
proving the second assertion.

Finally, if $n\geq k$, then 
\[
m_{k}^{1/k}\leq m_{n}^{1/n},\quad\Rightarrow\quad m_{k}^{n/k}\leq m_{n},
\]
while if $n<k,$then iterating \ref{eq:temp m_k^1/k reg 3} yields
\[
m_{k}^{1/k}\leq\exp\left[\log C\sum_{j=n+1}^{k}\frac{1}{j}\right]m_{n}^{1/n}\leq m_{n}^{1/n}\exp\left[\log C\left(\log k+1-\log n\right)\right].
\]
We choose $C_{4}>0$ so that 
\[
\exp\left[\log Cn\left(\log k+1-\log n\right)\right]\leq C_{4}^{n+k},
\]
and obtain the third assertion
\end{proof}
\begin{rem}
\label{rmk: slowly varying}In the last assertion of the above claim,
if in addition the sequence $v_{k}=m_{k}^{1/k}$ is slowly varying
(i.e., if $v_{2k}/v_{k}\to1$ as $k\to\infty)$, then $C_{4}$ can
be chosen arbitrarily close to one, i.e., \emph{for any $\eta>1$,
there exists $C_{\eta}>0$ so that
\[
m_{k}^{n/k}\leq C_{\eta}\eta^{k+n}m_{n}.
\]
}
\end{rem}

\begin{claim}
\label{claim:m_k reg 2nd}Let $C_{1}>1$ and $\eta>1.$ There exist
$k_{0}>0$ and $C_{3}>0$, such that for any $k_{0}\leq k\leq n\leq C_{1}k$,
the following hold:
\end{claim}

\begin{enumerate}
\item $\frac{1}{\eta^{n}}\leq\frac{\widehat{\gamma}_{n}}{n!}\left(\frac{1}{\eta\widehat{L}(k)}\right)^{n}\leq\frac{1}{\eta^{n/2}}.$
\item $1\leq m_{n}m_{k}^{-n/k}\leq C_{3}^{n}$
\end{enumerate}
\begin{proof}
Recall that 
\[
\frac{\widehat{\gamma}_{k}}{n!}\left(\frac{1}{\eta\widehat{L}(k)}\right)^{n}=\left(\frac{\widehat{L}(n)}{\eta\widehat{L}(k)}\right)^{n}.
\]
The function $\widehat{L}$ is eventually non-decreasing,
thus for $k>k_{0}$,
\[
\left(\frac{\widehat{L}(k)}{\eta\widehat{L}(k)}\right)^{n}\leq\left(\frac{\widehat{L}(n)}{\eta\widehat{L}(k)}\right)^{n}\leq\left(\frac{\widehat{L}(Ck)}{\eta L(k)}\right)^{n}
\]
Since $\frac{\widehat{L}(Ck)}{\widehat{L}\left(k\right)}\to1$, we
have
\[
\frac{1}{\eta^{n}}\leq\left(\frac{\widehat{L}(n)}{\eta\widehat{L}(k)}\right)^{n}\leq\frac{1}{\eta^{n/2}}
\]
proving the first assertion.

As for the second assertion, by the second assertion of Claim \ref{claim:m_k reg first claim},
we have 
\[
m_{n}m_{k}^{-n/k}=m_{n}\left(\frac{1}{m_{k}^{1/k}}\right)^{n}\leq C_{3}^{n}\frac{m_{k}^{n/k}}{m_{k}^{n/k}}=C_{3}^{n}.
\]
Since $m_{j}^{1/j}$ is non-decreasing, we have

\[
m_{n}m_{k}^{-n/k}=\left(\frac{m_{n}^{1/n}}{m_{k}^{1/k}}\right)^{n}\geq1,
\]
proving the second assertion.
\end{proof}
Fix a compact interval $J$ such that $0\in J$. Given a $z\in\mathbb{C},$
denote by $z^{*}$ its projection on $J$, i.e. $z^{*}\in J$ and
satisfies
\[
\left|z-z^{*}\right|=\text{dist}\left(z,J\right).
\]

Given $\eta>0$ and $Q>0$ consider the set 
\[
V_{k,\eta,Q}=\left\{ z:\text{dist}\left(z,J\right)\leq\frac{\left|z^{*}\right|\eta}{\widehat{L}(k)}\right\} \cup\left\{ z:\text{dist}\left(z,J\right)\leq\frac{1}{Qm_{k}^{1/k}}\right\} .
\]

\begin{figure}[H]

\includegraphics[scale=1.4]{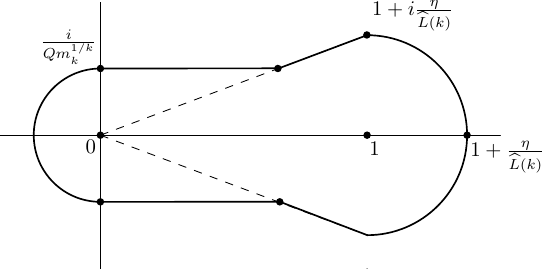}\caption{The set $V_{k,\eta,Q}$ for $J=[0,1]$}

\end{figure}

\begin{claim}
\label{claim:distance between V diff k}For fixed $\eta>0$ and $Q>0,$
we have
\[
\text{dist}\left(\partial V_{k+1,\eta,Q},\partial V_{k,\eta,Q}\right)\geq\min\left\{ \frac{1}{Qm_{k+1}^{1/\left(k+1\right)}}-\frac{1}{Qm_{k}^{1/k}}\,,\,\frac{\eta}{2}\frac{1}{Qm_{k+1}^{1/\left(k+1\right)}}\left(\frac{\widehat{L}(k+1)}{\widehat{L}(k)}-1\right)\right\} ,\quad k>k_{0}.
\]
\end{claim}

\begin{proof}
Without loss of generality assume that $J=[0,1]$. Let $z_{0}=x_{0}+iy_{0}\in\partial V_{k+1,\eta,Q}$.
Assume first that $x_{0}\leq1$. If $z_{0}\in\partial\left\{ z:\text{dist}\left(z,J\right)\leq\frac{1}{Qm_{k+1}^{1/\left(k+1\right)}}\right\} ,$then
because $\frac{\text{dist}\left(z,J\right)}{z^{*}}\geq\text{dist}\left(z,J\right)$,
we have 
\[
\text{dist}\left(z_{0},\partial V_{k,\eta,Q}\right)\geq\text{dist}\left(z_{0},\left\{ z:\text{dist}\left(z,J\right)\leq\frac{1}{Qm_{k}^{1/k}}\right\} \right)=\frac{1}{Qm_{k+1}^{1/\left(k+1\right)}}-\frac{1}{Qm_{k}^{1/k}}.
\]
On the other hand, if $z_{0}\in\partial\left\{ z:\frac{\text{dist}\left(z,J\right)}{z^{*}}\leq\frac{\eta}{\widehat{L}(k+1)}\right\} $,
then $\text{dist}\left(z_{0},\partial V_{k,\eta,Q}\right)$ is greater
than the distance from $z_{0}$ to the line $y=\frac{\eta}{\widehat{L}(k)}x,$
therefore
\[
\text{dist}\left(z_{0},\partial V_{k,\eta,Q}\right)\geq\frac{\left|\frac{\eta}{\widehat{L}(k)}x_{0}-y_{0}\right|}{\sqrt{1+\frac{\eta^{2}}{\widehat{L}^{2}(k)}}}.
\]
Observe that $z_{0}\in\partial\left\{ z:\frac{\text{dist}\left(z,J\right)}{z^{*}}\leq\frac{\eta}{\widehat{L}(k+1)}\right\} $
implies that
\[
x_{0}\geq\frac{1}{Qm_{k+1}^{1/\left(k+1\right)}}\frac{\widehat{L}(k+1)}{\eta},\quad\text{and\ensuremath{\quad y_{0}=x_{0}\frac{\eta}{\widehat{L}(k+1)}.}}
\]
Therefore
\[
\frac{\left|\frac{\eta}{\widehat{L}(k)}x_{0}-y_{0}\right|}{\sqrt{1+\frac{\eta^{2}}{\widehat{L}^{2}(k)}}}\geq\frac{\eta}{2}\frac{1}{Qm_{k+1}^{1/\left(k+1\right)}}\left(\frac{\widehat{L}(k+1)}{\widehat{L}(k)}-1\right)
\]
Finally, if $x_{0}>1$, then 
\[
\text{dist}\left(z_{0},J\right)=\max\left\{ \frac{1}{Qm_{k+1}^{1/\left(k+1\right)}},\frac{\eta}{\widehat{L}(k+1)}\right\} 
\]
and therefore
\begin{align*}
\text{dist}\left(z_{0},\partial V_{k,\eta,Q}\right) & \geq\max\left\{ \frac{1}{Qm_{k}^{1/k}},\frac{\eta}{\widehat{L}(k)}\right\} -\max\left\{ \frac{1}{Qm_{k+1}^{1/\left(k+1\right)}},\frac{\eta}{\widehat{L}(k+1)}\right\} \\
& \geq\max\left\{ \frac{1}{Qm_{k}^{1/k}}-\frac{1}{Qm_{k+1}^{1/\left(k+1\right)}}\,,\,\frac{\eta}{\widehat{L}(k)}-\frac{\eta}{\widehat{L}(k+1)}\right\} ,
\end{align*}
proving the claim.
\end{proof}
For $f\in C^{\infty}\left(J\right)$ and $N\in\mathbb{N},$ consider
the function 
\[
P_{N}\left(z\right)=P_{N}\left(z;f\right)=\sum_{j=0}^{N}\frac{f^{(j)}(z^{*})}{j!}(z-z^{*})^{j}.
\]
Clearly $P_{N}\in C^{\infty}\left(\mathbb{C}\right)$ and $P_{N}\vert_{J}=f.$
\begin{claim}
\label{claim:First d-bar estimate}Suppose that $f\in C^{\infty}\left(J\right)$
satisfies the bounds
\[
\left|f^{(n)}(x)\right|\leq C\min\left\{ \frac{\widehat{\gamma}_{n}}{\eta_{0}^{n}x^{n}},M_{n}\right\} ,\forall n\geq1.
\]
Then for any $0<\eta<\eta_{0}$ and any $Q_{1}>1$, there exists $Q>0$
and $\ell,k_{0}\in\mathbb{N}$ so that 
\[
\left|\bar{\partial}P_{\ell k-\ell}(z)\right|+\left|\bar{\partial}P_{\ell k}(z)\right|+\left|\bar{\partial}P_{\ell k+\ell}(z)\right|+\left|P_{\ell k+\ell}(z)-P_{\ell k-\ell}(z)\right|\leq\frac{C}{Q_{1}^{k}},
\]
for any $k>k_{0}$ and $z\in V_{k,\eta,Q}.$
\begin{proof}
Observe that 
\[
\left|\bar{\partial}P_{N}(z)\right|\leq\frac{1}{2}\max\left\{ m_{N+1}\text{dist}^{N}(z,J)\,,\,\frac{\widehat{\gamma}_{N+1}}{\eta_{0}^{N}\left(N+1\right)!}\left|\frac{z}{z^{*}}-1\right|^{N}\right\} .
\]
Therefore, if $z\in V_{k,\eta,Q}$, then
\[
\left|\bar{\partial}P_{N}(z)\right|\leq\frac{1}{2}\max\left\{ m_{N+1}\left(\frac{1}{Qm_{k}^{1/k}}\right)^{N}\,,\,\widehat{\gamma}_{N+1}\left(\frac{\eta}{\eta_{0}}\cdot\frac{1}{\widehat{L}(k)}\right)^{N}\right\} 
\]
By the first assertion of Claim \ref{claim:m_k reg first claim},
we have 
\[
m_{N+1}\leq Cm_{N}N^{C_{2}},\quad\frac{\widehat{\gamma}_{N+1}}{\left(N+1\right)!}\leq C\frac{\widehat{\gamma}_{N}}{N!}N^{C_{2}},
\]
and therefore
\[
\left|\bar{\partial}P_{N}(z)\right|\leq CN^{C_{2}}\cdot\max\left\{ m_{N}\left(\frac{1}{Qm_{k}^{1/k}}\right)^{N}\,,\,\frac{\widehat{\gamma}_{N}}{N!}\left(\frac{\eta}{\eta_{0}}\cdot\frac{1}{\widehat{L}(k)}\right)^{N}\right\} .
\]
In particular, if $N=\ell k-\ell,\,\ell k,\,\ell k+\ell$ and $Q$
is sufficiently large, then Claim \ref{claim:m_k reg 2nd} yields
\[
\left|\bar{\partial}P_{N}(z)\right|\leq C\ell^{C_{2}}\left(k+1\right)^{C_{2}}\cdot\max\left\{ \left(\frac{C_{6}}{Q}\right)^{\left(k-1\right)\ell}\,,\,\left(\frac{\eta}{\eta_{0}}\right)^{\frac{1}{2}\left(k-1\right)\ell}\right\} .
\]
Therefore, we may choose $\ell$ and $Q$ large enough that
\[
\left|\bar{\partial}P_{\ell k-\ell}(z)\right|+\left|\bar{\partial}P_{\ell k}(z)\right|+\left|\bar{\partial}P_{\ell k+\ell}(z)\right|\leq\frac{1}{2}\frac{1}{Q_{1}^{k}}.
\]
Next, we deal with the term $\left|P_{\ell k+\ell}(z)-P_{\ell k-\ell}(z)\right|$.
By the triangle inequality we get
\begin{align*}
\left|P_{\ell k+\ell}(z)-P_{\ell k-\ell}(z)\right| & \leq\sum_{j=\ell k-\ell}^{\ell k+\ell}\left|\frac{f^{(j)}(z^{*})}{j!}\right|\left|z-z^{*}\right|^{j}\\
& \leq\max\left\{ \sum_{j=\ell k-\ell}^{\ell k+\ell}m_{j}Q^{-k}m_{k}^{-j/k}\,,\,\sum_{j=\ell k-\ell}^{\ell k+\ell}\frac{\widehat{\gamma}_{j}}{j!}\left(\frac{1}{\eta\widehat{L}(k)}\right)^{j}\right\} .
\end{align*}
Therefore, by Claim \ref{claim:m_k reg 2nd}, we have 
\[
\left|P_{\ell k+\ell}(z)-P_{\ell k-\ell}(z)\right|\leq\left(2\ell+1\right)\max\left\{ \left(\frac{C_{6}}{Q}\right)^{\left(k-1\right)\ell}\,,\left(\frac{\eta}{\eta_{0}}\right)^{\frac{1}{2}\left(k-1\right)\ell}\right\} \leq\frac{1}{2}\frac{1}{Q_{1}^{k}}.
\]
\end{proof}
\end{claim}

\begin{lem}
\label{lem: d-bar full}Suppose for some $\delta>0$, $\frac{1}{m_{k}^{1/k}}-\frac{1}{m_{k+1}^{1/\left(k+1\right)}}\gtrsim e^{-\delta k}$
and $\widehat{L}(k+1)-\widehat{L}\left(k\right)\gtrsim e^{-\delta k}.$
Assume further that $f\in C^{\infty}\left(J\right)$ satisfies the
bounds
\[
\left|f^{(n)}(x)\right|\leq C\min\left\{ \frac{\widehat{\gamma}_{n}}{\eta_{0}^{n}x^{n}},M_{n}\right\} ,\forall n\geq1.
\]
Then for any $0<\eta<\eta_{0}$ and any $Q_{2}>1$, there exists $k_{0},\,Q>0$
and $F\in C_{C}^{1}\left(\mathbb{C}\right)$ so that $F\vert_{J}=f$
and 
\[
\left|\bar{\partial}F(z)\right|\leq\frac{1}{Q_{2}^{k}},\quad z\in V_{k,\eta,Q},\,k\geq k_{0}.
\]
\end{lem}

\begin{proof}
For each $k,$ let $\phi_{k}\in C_{C}^{1}(\mathbb{C})$ be so that
$\phi_{k}\equiv1$ in $V_{k,\eta,Q}\setminus V_{k+1,\eta,Q}$ and
zero outside $V_{k-1,\eta,Q}\setminus V_{k+2,\eta,Q}$. We construct
these functions $\phi_{k}$ inductively in such a way that

\[
\bar{\partial}\phi_{k-1}=-\bar{\partial}\phi_{k+1},\quad\text{in \ensuremath{V_{k,\eta,Q}\setminus V_{k+1,\eta,Q}}}.
\]
By Claim \ref{claim:distance between V diff k} and our assumption,
we may choose these $\left(\phi_{k}\right)$ so that 
\[
\left|\bar{\partial}\phi_{k-1}\right|\lesssim e^{2\delta k},\quad\text{in }V_{k,\eta,Q}\setminus V_{k+1,\eta,Q},\quad\forall\delta>0.
\]
For $1<\eta<\eta_{0}$ and $Q_{1}=e^{3\delta}Q_{2}$, we choose $\ell$
and $Q$ as in Claim \ref{claim:First d-bar estimate} and put
\[
F(z)=\sum_{j=1}^{\infty}\phi_{j}\left(z\right)P_{j\ell}(z).
\]
By definition, if $z\in V_{k,\eta,Q}\setminus V_{k+1,\eta,Q}$, then
\[
F(z)=\phi_{k-1}(r)P_{\ell\left(k-1\right)}(z)+\phi_{k}(r)P_{\ell k}(z)+\phi_{k+1}(r)P_{\ell\left(k+1\right)}(z).
\]
Therefore
\begin{align*}
\left|\bar{\partial}F(z)\right| & \leq\left|\bar{\partial}P_{\ell\left(k-1\right)}(z)\right|+\left|\bar{\partial}P_{\ell k}(z)\right|+\left|\bar{\partial}P_{\ell\left(k+1\right)}(z)\right|+\left|\bar{\partial}\phi_{k-1}\right|\left|P_{\ell\left(k+1\right)}(z)-P_{\ell\left(k-1\right)}(z)\right|\\
& \leq\frac{e^{3\delta k}}{Q_{1}^{k}}=\frac{1}{Q_{2}^{k}}.
\end{align*}
Observing that the above bound is decreasing in $k$ and that 
\[
V_{k,\eta,Q}=\bigcup_{j\geq k}V_{j,\eta,Q}
\]
finishes the proof.
\end{proof}

\subsubsection{Approximation by analytic functions of functions in $B(\widehat{\gamma},M;I)$:
the non-analytic case}
\begin{lem}
Suppose for some $\delta>0,$ $\frac{1}{m_{k}^{1/k}}-\frac{1}{m_{k+1}^{1/\left(k+1\right)}}\gtrsim e^{-\delta k}$
and $\widehat{L}(k+1)-\widehat{L}\left(k\right)\gtrsim e^{-\delta k}.$
Assume further that $f\in B_{1/\eta_{0}}(M,\widehat{\gamma};J)$.
Then for any $Q_{3}>0$ and $0<\eta<\eta_{0}$, there exists $Q>0$
and a sequence $\left(F_{k}\right)_{k\geq k_{0}}$ so that 
\end{lem}

\begin{enumerate}
\item $F_{k}\in\text{Hol}\left(V_{k,\eta,Q}\right)\cup C\left(\overline{V_{k,\eta,Q}}\right),$
with $\max_{V_{k,\eta,Q}}\left|F_{k-1}-F_{k}\right|\leq\frac{1}{Q_{3}^{k}}$.
\item $F_{k}\to f$ in the $C^{\infty}\left(J\right)$ topology, i.e., $F_{k}^{(n)}\to f^{(n)}$
uniformly in $J$ for any $n\geq0.$
\end{enumerate}
\begin{proof}
We first assume that $f\in C_{1/\eta_{0}}(M,\widehat{\gamma};J).$
Replacing $M_{n}$ with $C^{n}M_{n}$, we may assume that 
\[
\left|f^{(n)}(x)\right|\leq C\min\left\{ \frac{\widehat{\gamma}_{n}}{\eta_{0}^{n}x^{n}},M_{n}\right\} ,\forall n\geq1.
\]
For $1<\eta<\eta_{0}$ and $Q_{2}=2Q_{3},$ let $F$ be the function
constructed in Lemma \ref{lem: d-bar full}. By the Cauchy--Pompeiu
formula, 
\[
F(z)=-\frac{1}{\pi}\int_{\mathbb{C}}\bar{\partial}F(w)\frac{1}{w-z}d\mu_{2}(w)
\]
where $\mu_{2}$ is the planar Lebesgue measure in $\mathbb{C}$.
Given $k\geq k_{0}$, consider the functions 
\[
F_{k}(z)=\int_{\mathbb{C\setminus}V_{k-1,\eta_{0},Q}}\bar{\partial}F(w)\frac{1}{w-z}d\mu_{2}(w)
\]
Since $\cap_{k}V_{k-1,\eta_{0},Q}=J$, we have $F_{k}\to F=f$ in
the $C^{\infty}\left(J\right)$ topology. Moreover, for $k\geq k_{0},$
the function $F_{k}$ is analytic in $\overline{V_{k,\eta_{0},Q}}\subset V_{k-1,\eta_{0},Q}$
. Writing 
\[
F_{k-1}-F_{k}=\int_{V_{k-2,\eta_{0},Q}\mathbb{\setminus}V_{k-1,\eta_{0},Q}}\bar{\partial}F(w)\frac{1}{w-z}d\mu_{2}(w)
\]
and recalling that

\[
\left|\bar{\partial}F\right|\leq\frac{1}{Q_{2}^{k}},\quad\text{in}\,V_{k-2,\eta_{0},Q},
\]
we have
\[
\left|F_{k-1}-F_{k}\right|\leq\frac{1}{Q_{2}^{k}}\frac{1}{\text{dist\ensuremath{\left(\partial V_{k,\eta_{0},Q},\partial V_{k-1,\eta_{0},Q}\right)}}}\quad\text{in}\,\,V_{k,\eta_{0},Q}.
\]
Thus, by Claim \ref{claim:distance between V diff k}, we have
\[
\left|F_{k-1}-F_{k}\right|\leq\frac{e^{3\delta k}}{Q_{2}^{k}}=\frac{1}{Q_{3}^{k}}\quad\text{in}\,\,V_{k,\eta_{0},Q},
\]
proving the claim for $f\in C_{1/\eta_{0}}(M,\widehat{\gamma};J)$.
By Lemma \ref{lem:non-homg and B}, for $\eta<\eta_{1}<\eta_{0}$,
we have $B_{1/\eta_{0}}(M,\widehat{\gamma};J)\subseteq C_{1/\eta_{1}}(M,\widehat{\gamma};J)$
and the claim follows. 
\end{proof}
For a given $\eta>1$ and $Q>0,$ denote by $U_{k,\eta,Q}$ a $\frac{1}{Qm_{k}^{1/k}}$
neighborhood of the two sectors 
\[
\left\{ z:\text{arg}(z)\leq\frac{\eta}{\widehat{L}(k)},\,\quad0<\text{Re}\left(z\right)\in J\right\} \bigcup\left\{ z:\text{arg}(-z)\leq\frac{\eta}{\widehat{L}(k)},\,\quad0>\text{Re}\left(z\right)\in J\right\} .
\]
\begin{figure}[H]
\includegraphics[scale=1.4]{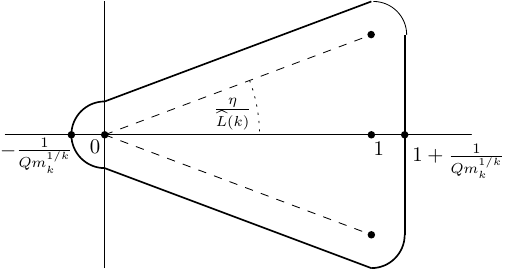}\label{fig: The set U}\caption{The set $U_{k,\eta,Q}$ for $J=[0,1]$}
\end{figure}
Observe that for any $\eta_{1}>\eta>1$ and $Q_{1}>1$, there exist
$Q>1$ so that 
\[
U_{k,\eta,Q}\subset V_{k,\eta_{1},Q_{1}},\quad\forall k\geq k_{0}.
\]
Thus, the previous Lemma immediately give us the following one:
\begin{lem}
\label{lem: Aprox by analytic function 2nd}Suppose for some $\delta>0$,
$\frac{1}{m_{k}^{1/k}}-\frac{1}{m_{k+1}^{1/\left(k+1\right)}}\gtrsim e^{-\delta k}$
, $\widehat{L}(k+1)-\widehat{L}\left(k\right)\gtrsim e^{-\delta k},$
and that $f\in B_{1/\eta_{0}}(M,\widehat{\gamma};J)$. Then for any
$Q_{3}>0$ and $0<\eta<\eta_{0}$, there exists $Q>0$ and a sequence
$\left(F_{k}\right)_{k\geq k_{0}}$ so that 
\end{lem}

\begin{enumerate}
\item $F_{k}\in\text{Hol}\left(U_{k,\eta,Q}\right)\cup C\left(\overline{U_{k,\eta,Q}}\right),$
with $\max_{U_{k,\eta,Q}}\left|F_{k-1}-F_{k}\right|\leq\frac{1}{Q_{3}^{k}}$.
\item $F_{k}\to f$ in the $C^{\infty}\left(J\right)$ topology, i.e., $F_{k}^{(n)}\to f^{(n)}$
uniformly in $J$ for any $n\geq0.$
\end{enumerate}
\begin{rem}
In what follows, we will apply Lemma \ref{lem: Aprox by analytic function 2nd}
to functions in $B_{1/\eta_{0}}(M\cdot\gamma,\widehat{\gamma};0_{+})$,
with $M$ regular and $\gamma$ admissible. The conditions
\[
\widehat{L}(k+1)-\widehat{L}\left(k\right)\gtrsim e^{-\delta k},\quad\frac{1}{\gamma_{k}^{1/k}m_{k}^{1/k}}-\frac{1}{\gamma_{k+1}^{1/\left(k+1\right)}m_{k+1}^{1/\left(k+1\right)}}\gtrsim e^{-\delta k}
\]
follow from our admissibility assumption. Indeed, the former left-hand side inequality follows from Admissibility assumption (D), while
the right hand side inequality follows from
\begin{align*}
\frac{1}{\gamma_{k}^{1/k}m_{k}^{1/k}}-\frac{1}{\gamma_{k+1}^{1/\left(k+1\right)}m_{k+1}^{1/\left(k+1\right)}} & \geq\frac{1}{m_{k+1}^{1/\left(k+1\right)}}\left(\frac{1}{\gamma_{k}^{1/k}}-\frac{1}{\gamma_{k+1}^{1/\left(k+1\right)}}\right)=\frac{1}{m_{k+1}^{1/\left(k+1\right)}}\int_{k}^{k+1}\frac{1}{L(x)\cdot\widehat{L}(x)}\cdot\frac{dx}{x}\\
& \gtrsim\frac{1}{m_{k+1}^{1/\left(k+1\right)}}\frac{1}{L(k+1)\widehat{L}(k+1)}\frac{1}{k}\gtrsim\frac{1}{m_{k+1}^{1/\left(k+1\right)}}k^{-C}.
\end{align*}
By the first assertion of Claim \ref{claim:m_k reg first claim},
the right hand side is $\gtrsim k^{-C_{1}}$ for some constant $C_{1}>0.$ 
\end{rem}

\subsection{Almost holomorphic extensions of functions in $B(M,\widehat{\gamma};I)$:
the analytic case}

Throughout this section, we fix an admissible $\gamma$ with an analytic
$\widehat{\gamma}$ and a regular sequence $M.$ Recall that 
\[
\Omega_{\eta}\left(\gamma\right):=\left\{ z:\sup_{t>0}E(t\eta)\left|K\left(\frac{t}{z}\right)\right|<\infty\right\} ,
\]
and
\[
D\left(M,\gamma;[0,a)\right)=\bigcup_{\eta>a}C_{\infty}^{M}\left(\text{\ensuremath{\Omega_{\eta}}}\mathbb{\left(\gamma\right)}\right).
\]
We will use the following Theorem of Dynkin \cite{dyn1976pseudoanalytic}.
\begin{thm*}[Dynkin]
Let $\mathbb{E}\subset\mathbb{C}$ be a perfect compact set. If $f\in C_{\infty}^{M}\left(\mathbb{E}\right),$
then there is $C>0$ and a function $F\in C_{C}^{1}\left(\mathbb{C}\right)$
so that $F\vert_{\mathbb{E}}=f$ so that 
\[
\left|\bar{\partial}F(z)\right|\leq C\cdot h\left(C\cdot\text{dist}\left(z,\mathbb{E}\right)\right),z\in\mathbb{C},
\]
where $h(r)=\inf_{n\geq0}m_{n}r^{n}$.
\end{thm*}
\begin{lem}
\label{lem:aprox by analytic function-the analytic case}Let $f\in C_{\infty}^{M}\left(\mathbb{E}\right)$.
Then for any $Q_{1}>0$ there is $Q>0$ and a sequence $\left(F_{k}\right)_{k\geq k_{0}}$
so that 
\end{lem}

\begin{enumerate}
\item the functions $F_{k}$ are analytic in $\left\{ z:\text{dist}\left(z,\mathbb{E}\right)\leq\frac{1}{Q^{k}m_{k}^{1/k}}\right\} $
and satisfy therein $\left|F_{k-1}-F_{k}\right|\leq\frac{1}{Q_{1}^{k}}$
for any $k\geq k_{0};$
\item $F_{k}\to f$ in the $C^{\infty}\left(\mathbb{E}\right)$ topology,
i.e., $F_{k}^{(n)}\to f^{(n)}$ uniformly in $J$ for any $n\geq0.$
\end{enumerate}
\begin{proof}
Let $Q_{1}>0$ and let $F$ be a function as in Dynkin's Theorem.
By the Cauchy--Pompeiu formula, 
\[
F(z)=-\frac{1}{\pi}\int_{\mathbb{C}}\bar{\partial}F(w)\frac{1}{w-z}d\mu_{2}(w)
\]
where $\mu_{2}$ is the planar Lebesgue measure in $\mathbb{C}$.
For $Q>0$ that will be chosen later, put
\[
\widetilde{U}_{k}:=\left\{ z:\text{dist}\left(z,\mathbb{E}\right)\leq\frac{1}{Q^{k}m_{k}^{1/k}}\right\} 
\]
and
\[
F_{k}(z)=\int_{\mathbb{C\setminus}\widetilde{U}_{k-1}}\bar{\partial}F(w)\frac{1}{w-z}d\mu_{2}(w).
\]
Since $\cap_{k}\widetilde{U}_{k}=\mathbb{E}$, we have $F_{k}\to F=f$
in the $C^{\infty}\left(\mathbb{E}\right)$ topology.

If $w\in\widetilde{U}_{k-2}$, then 
\[
\left|\bar{\partial}F(w)\right|\leq C\cdot\inf_{n}m_{n}\left(\frac{C}{Q^{k-2}m_{k-2}^{1/\left(k-2\right)}}\right)^{n}\stackrel{\left(n=k-2\right)}{=}\frac{C^{k-1}}{Q^{k-2}}.
\]
Therefore we can choose $Q$ large enough so that 
\[
\left|\bar{\partial}F(w)\right|\leq\frac{1}{Q_{1}^{k}},\quad w\in\widetilde{U}_{k-2},\,k\geq k_{0}.
\]
Writing
\[
F_{k-1}-F_{k}=\int_{\widetilde{U}_{k-2}\setminus\widetilde{U}_{k-1}}\bar{\partial}F(w)\frac{1}{w-z}d\mu_{2}(w),
\]
we obtain
\[
\left|F_{k-1}-F_{k}\right|\leq\frac{1}{Q_{2}^{k}}\quad\text{in\,}\,\widetilde{U}_{k}.
\]
\end{proof}

\subsection{The classes $B^{\pm}(\widehat{\gamma},M;J)$\label{subsec:The-classes plus minus}}

Throughout this section we fix an admissible $\gamma$ and a regular
sequence $M.$

\subsubsection{Almost holomorphic extension of functions in $B^{\pm}(\widehat{\gamma},M;0)$ }

For $\mu>0$ and a compact $J$, put $S_{\mu}^{\pm}=\left\{ z:\text{\ensuremath{\pm}Im}z>0,\,|z|<\mu\right\} .$
Let $f_{+}\in C^{\infty}\left(\overline{S_{2\mu}^{+}}\right)\cap\text{Hol}\left(S_{2\mu}^{+}\right)$
be such that 
\[
\left|f_{+}^{(n)}(x)\right|\leq C\min\left\{ \frac{\widehat{\gamma}_{n}}{\eta_{0}^{n}x^{n}}\,,\,\,M_{n}\right\} ,\qquad\forall n\geq1,\,\,x\in[-2\mu,2\mu].
\]
Put

\[
P_{N}^{+}\left(z\right)=P_{N}^{+}\left(z;f\right)=\sum_{j=0}^{N}\frac{f^{(j)}(z^{*})}{j!}(z-z_{\mu}^{*})^{j},
\]
where $z_{\mu}^{*}$ is the projection of $z$ onto $\overline{S_{\mu}^{+}}$.
Observe that as
\[
\left|\bar{\partial}P_{N}^{+}(z)\right|\leq\frac{1}{2}\max\left\{ m_{N+1}\text{dist}^{N}(z,J)\,,\,\frac{\widehat{\gamma}_{N+1}}{\eta_{0}^{N}\left(N+1\right)!}\left|\frac{z}{z^{*}}-1\right|^{N}\right\} ,\quad\forall z\in\mathbb{C}.
\]
and that $\bar{\partial}P_{N}^{+}=0$ in $\overline{S_{\mu}^{+}}.$ Thus,
repeating all the steps in the proof of Lemma \ref{lem: Aprox by analytic function 2nd},
we obtain the following analogue.
\begin{lem}
\label{lem:analytic aprox for =00005Cplus case}Suppose for some $\delta>0$,
$\frac{1}{m_{k}^{1/k}}-\frac{1}{m_{k+1}^{1/\left(k+1\right)}}\gtrsim e^{-\delta k}$
, $\widehat{L}(k+1)-\widehat{L}\left(k\right)\gtrsim e^{-\delta k},$
and that $f\in B_{1/\eta_{0}}^{+}(M,\widehat{\gamma};[-1,1])$. Then
for any $Q_{3}>0$ and $0<\eta<\eta_{0}$, there exist $Q>0$, $\mu>0$
and a sequence $\left(F_{k}\right)_{k\geq k_{0}}$ so that 
\end{lem}

\begin{enumerate}
\item $F_{k}\in\text{Hol}\left(\overline{\left(\left\{ \left|\text{Re}z\right|\leq\mu\right\} \cap U_{k,\eta,Q}\right)\cup S_{\mu}^{+}}\right),$
with $\max_{\left(\left\{ \left|\text{Re}z\right|\leq\mu\right\} \cap U_{k,\eta,Q}\right)\cup S_{\mu}^{+}}\left|F_{k-1}-F_{k}\right|\leq\frac{1}{Q_{3}^{k}}$.
\item $F_{k}\to f$ in the $C^{\infty}\left(\left[-\mu,\mu\right]\right)$
topology, i.e., $F_{k}^{(n)}\to f^{(n)}$ uniformly in $\left[-\mu,\mu\right]$
for any $n\geq0.$
\end{enumerate}

\subsubsection{The factorization $B_{\eta}(M,\widehat{\gamma};0)\subseteq B_{\eta_{1}}^{+}(M,\widehat{\gamma};0)+B_{\eta_{1}}^{-}(M\gamma,\widehat{\gamma};0)$}

In this section we show that $B_{\eta}(M,\widehat{\gamma};0)\subseteq B_{\eta_{1}}^{+}(M,\widehat{\gamma};0)+B_{\eta_{1}}^{-}(M,\widehat{\gamma};0)$
where $0<\eta<\eta_{1},$ $M$ is regular, and $\gamma$ is admissible
with non-analytic $\widehat{\gamma}.$
\begin{proof}
Let $0<\eta<\eta_{0}<\eta_{1}$. By Lemma \ref{lem: d-bar full}, there
exists $\mu>0$ such that for any $Q_{2}>1$, there exists $k_{0},\,Q>0$
and $F\in C_{C}^{1}\left(\mathbb{C}\right)$ so that $F\vert_{[-2\mu,2\mu]}=f$
and 
\[
\left|\bar{\partial}F(z)\right|\leq\frac{1}{Q_{2}^{k}},\quad z\in\left\{ z:\text{dist}\left(z,\left[-2\mu,2\mu\right]\right)\leq\frac{\left|z^{*}\right|}{\widehat{L}(k)\eta_{0}}\right\} \cup\left\{ z:\text{dist}\left(z,\left[-2\mu,2\mu\right]\right)\leq\frac{1}{Qm_{k}^{1/k}}\right\} .
\]
By the Cauchy--Pompeiu formula, 
\[
F(z)=-\frac{1}{\pi}\int_{\mathbb{C}}\bar{\partial}F(w)\frac{1}{w-z}d\mu_{2}(w)
\]
Put 
\[
f_{\pm}(z)=-\frac{1}{\pi}\int_{\mathbb{C_{\pm}}}\bar{\partial}F(w)\frac{1}{w-z}d\mu_{2}(w),
\]
where $\mathbb{C}_{\pm}=\left\{ z:\text{\ensuremath{\pm}Im}z>0\right\} .$
Clearly $f=f_{+}+f_{-}$ and the functions $f_{\pm}$ are analytic
in $S_{\mu}^{\pm}:=\left\{ z:\pm\text{Im}z>0,\,|z|\leq\mu\right\} $,
thus it is sufficient to prove that $f_{+}\in B_{\eta_{1}}(M\gamma,\widehat{\gamma};0).$
Observe that 
\begin{equation}
\left|f_{+}^{(n)}(x)\right|\leq n!\int_{\mathbb{C_{+}}}\left|\bar{\partial}F(w)\right|\frac{1}{\left|w-x\right|^{n+1}}d\mu_{2}(w),\quad x\in[-\mu,\mu].\label{eq: f_+ der}
\end{equation}
If 
\[
\text{\ensuremath{\frac{1}{Qm_{k}^{1/k}}}\ensuremath{\ensuremath{\leq}}dist}\left(w,\left[-2\mu,2\mu\right]\right)\leq\frac{1}{Qm_{k-1}^{1/k-1}},
\]
then
\[
\left|\bar{\partial}F(w)\right|\frac{1}{\left|w-x\right|^{n+1}}\leq\frac{1}{Q_{2}^{k-1}}\left(Q^{n}m_{n}^{n/k}\right)
\]
Thus
\[
\left|f_{+}^{(n)}(x)\right|\leq Cn!\sup_{k\geq k_{0}}\frac{1}{Q_{2}^{k-1}}\left(Q^{n}m_{n}^{n/k}\right)
\]
By Claim \ref{claim:m_k reg first claim}, we can choose $Q_{2}$
so enough to $1$ that 
\[
\left|f_{+}^{(n)}(x)\right|\leq C^{n+1}M_{n}.
\]
Similarly using (\ref{eq: f_+ der}) again, for $w$ satisfying

\[
\frac{\left|w^{*}\right|}{\widehat{\eta_{0}L}(k)}\text{\ensuremath{\leq}dist}\left(w,\left[-2\mu,2\mu\right]\right)\leq\frac{\left|w^{*}\right|}{\eta_{0}\widehat{L}(k-1)},
\]
we obtain
\[
\left|f_{+}^{(n)}(x)\right|\leq|x|^{-n}Cn!\frac{1}{Q_{2}^{k-1}}\left(\eta_{0}^{-n}\widehat{L}(k)^{n}\right)\cdot
\]
The function $\widehat{L}$ is slowly varying, so by Claim \ref{claim:m_k reg first claim}
and Remark \ref{rmk: slowly varying}, we can choose $Q_{2}$ so close
to one, such that
\[
\left|f_{+}^{(n)}(x)\right|\leq Cn!\frac{\widehat{L}(n)^{n}}{|x|^{n}}=C\frac{\widehat{\gamma}_{n}}{|x|^{n}}
\]
we have established $f_{+}\in C_{\eta_{1}}^{+}(M\gamma,\widehat{\gamma};0)$
for any $\eta_{1}>\eta.$ By Lemma \ref{lem:non-homg and B}, $f_{+}\in B_{\eta_{1}}^{+}(M\gamma,\widehat{\gamma};0)$
for any $\eta_{1}>\eta.$ The proof of $f_{-}\in B_{\eta_{1}}^{-}(M\gamma,\widehat{\gamma};0)$
is obtained in the same way.
\end{proof}

\section{Estimates for the generalized Borel transform}

Throughout this section we will use the following auxiliary lemmas.
Their proofs are given in Appendix $A.$
\begin{lem}
\label{lem: E curve lemma}Let $\gamma$ be an admissible weight.
For any $\eta>1$, there is $C>0$ so that for any $r>0$
\[
\max\left\{ \left|E(z)\right|:\text{\ensuremath{\left|\text{arg}(z)\right|\geq\frac{\eta}{\widehat{L}\left(L^{-1}(r)\right)}}\,\ensuremath{or}\,}|z|\leq r\right\} \leq CE(r).
\]
\end{lem}

\begin{lem}
\label{lem:E exp bound}Let $\gamma$ be an admissible weight. We
have 
\[
E\left(L(k)\right)\lesssim e^{2k},\quad k\geq1,
\]
where $L(k)=\text{\ensuremath{\gamma\left(k\right)^{1/k}}}.$
\end{lem}

\subsection{Estimates for the generalized Borel transform in the non-analytic
case\label{subsec:2nd inc}}

Throughout this section we fix an admissible $\gamma$ with non-analytic
$\widehat{\gamma}$ (i.e. $\lim_{\rho\to\infty}\varepsilon\left(\rho\right)=0)$,
a regular sequence $M=\left(n!m_{n}\right)_{n\geq0}$ and $J=[0,1]$.
Recall that $\gamma_{k}^{1/k}=L(k)$ and $\left(\widehat{\gamma}_{k}/k!\right)^{1/k}=\widehat{L}(k)$.
For a given $k\geq1$, $\eta>1$ and $Q>0$, recall that the set $U_{k,\eta,Q}$
(see Figure 5.2) was defined as a $\frac{1}{QL(k)m_{k}^{1/k}}$
neighborhood of the sector
\[
\left\{ z:\left|\text{arg}(z)\right|\leq\frac{\eta}{\widehat{L}(k)},\,\quad\text{Re}\left(z\right)\in J\right\} .
\]

The main estimate for generalized Borel transform is given in the
following lemma.
\begin{lem}
\label{lem:Main estimate genrlize borel}For any $1<\eta$ and $0<Q$,
there exists a constant $C=C(\eta,Q),$ such that if $g\in\text{Hol}\left(U_{k,\eta,Q}\right)\cap C\left(\overline{U_{k,\eta,Q}}\right)$
with $\Vert g\Vert_{L^{\infty}\left(U_{k,\eta,Q}\right)}\leq1,$ then
\[
\left|\left(\bbg g\right)^{\left(n\right)}\left(x\right)\right|\leq C^{n+1}\cdot n!\cdot m_{k}^{n/k}\cdot E\left(L\left(k\right)\right)E\left(x\right),
\]
for any $x\geq0$ and $n\in\mathbb{Z}_{+}$.
\end{lem}

\begin{proof}
For $0<R<R^{\prime},$ consider the curve (see Figure \ref{fig: Curve Gamma(R,R')})

\begin{figure}[H]

\includegraphics{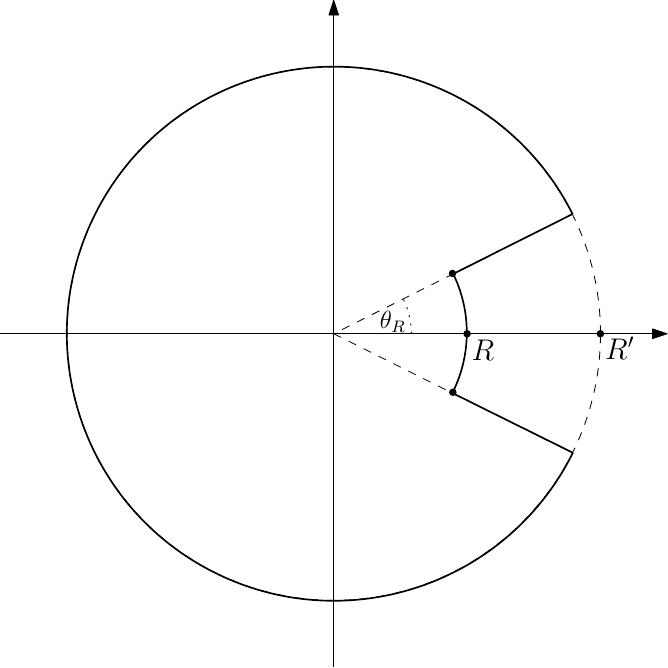}\label{fig: Curve Gamma(R,R')}\caption{The curve $\Gamma(R,R^{\prime})$}

\end{figure}

\[
\Gamma(R,R^{\prime}):=\left\{ Re^{i\theta}:\left|\theta\right|<\theta_{R}\right\} \cup\left\{ te^{\pm i\theta_{R}}:t\in\left[R,R^{\prime}\right]\right\} \cup\left\{ R^{\prime}e^{i\theta}:\theta_{R}<\left|\theta\right|\le\pi\right\} 
\]
where
\[
\theta_{R}=\frac{\eta}{\widehat{L}\left(L^{-1}\left(R\right)\right)}.
\]
Put $G=\bbg g$. By Cauchy formula, 
\[
G\left(x\right)=\frac{1}{2\pi i}\int_{\Gamma(R,R^{\prime})}g\left(\frac{x}{w}\right)E\left(w\right)\frac{dw}{w},
\]
provided that
\[
\left\{ \frac{x}{w}:w\in\Gamma(R,R^{\prime})\right\} \subset U_{k,\eta,Q}.
\]
Moreover, by differentiating under the integral sign, we find that
\[
G^{\left(n\right)}\left(x\right)=\frac{1}{2\pi i}\int_{\Gamma(R,R^{\prime})}g^{\left(n\right)}\left(\frac{x}{w}\right)E\left(w\right)\frac{dw}{w^{n+1}}.
\]
Taking $R^{\prime}\to\infty$ and putting
\[
\Gamma(R):=\left\{ Re^{i\theta}:\left|\theta\right|<\theta_{R}\right\} \cup\left\{ ue^{\pm i\theta_{R}}:u>R\right\} .
\]
We obtain
\[
G^{\left(n\right)}\left(x\right)=\frac{1}{2\pi i}\int_{\Gamma(R)}g^{\left(n\right)}\left(\frac{x}{w}\right)E\left(w\right)\frac{dw}{w^{n+1}}.
\]
By Cauchy's estimates, if 
\begin{equation}
\left\{ \frac{x}{w}:w\in\Gamma(R)\right\} \subseteq\left\{ z:\text{arg}(z)\leq\frac{\eta}{\widehat{L}(k)},\,\quad\text{Re}\left(z\right)\in[0,1]\right\} .\label{eq:inclu Gamma in V}
\end{equation}
then 
\[
\left|G^{\left(n\right)}\left(x\right)\right|\leq n!\left(Q\cdot m_{k}^{1/k}\cdot L(k)\right)^{n}\frac{1}{2\pi}\int_{\Gamma(R)}\left|E\left(w\right)\frac{dw}{w^{n+1}}\right|\leq n!\left(Q\cdot m_{k}^{1/k}\cdot L(k)\right)^{n}\frac{1}{R^{n}}\sup_{\Gamma(R)}\left|E\left(w\right)\right|.
\]
By Lemma \ref{lem: E curve lemma}, $\sup_{\Gamma(R)}\left|E\left(w\right)\right|\leq CE\left(R\right),$
so we have 
\[
\left|G^{\left(n\right)}\left(x\right)\right|\lesssim n!\left(Q\cdot m_{k}^{1/k}\cdot L(k)\right)^{n}\frac{E\left(R\right)}{R^{n}},
\]
provided that (\ref{eq:inclu Gamma in V}) holds. Finally we choose
\[
R=\max\left\{ L\left(k\right),x\right\} .
\]
and notice that with this choice 
\[
\theta_{R}=\frac{\eta}{\widehat{L}\left(L^{-1}\left(R\right)\right)}\leq\frac{\eta}{\widehat{L}\left(k\right)}.
\]
In particular, (\ref{eq:inclu Gamma in V}) holds. We have established
\[
\left|G^{\left(n\right)}\left(x\right)\right|\lesssim C^{n}\cdot n!\cdot m_{k}^{n/k}\max\left\{ E\left(L(k)\right),E\left(x\right)\right\} ,
\]
as claimed.
\end{proof}
\begin{proof}[Proof of the inclusion $\bbg B_{1/\eta}(M\gamma,\widehat{\gamma};I)\subseteq A(M,\gamma,I)$
in the non-analytic case]
It suffices to prove the inclusion in the case that $I$ is compact,
which by rescaling can be assumed to be $J=[0,1].$ Let $1<\eta<\eta_{0}$,
$f\in B_{1/\eta_{0}}(M\gamma,\widehat{\gamma};J)$, and let $C_{4}>0$
be as in Claim \ref{claim:m_k reg first claim}.

By Lemma \ref{lem: Aprox by analytic function 2nd} there is a $Q>0$
and a sequence $\left(f_{k}\right)_{k\geq k_{0}}\subset C^{\infty}\left(J\right)$
so that $\sum_{k\geq k_{0}}f_{k}$ converges to $f$ in the $C^{\infty}\left(J\right)$
topology, and $f_{k}\in\text{Hol}\left(U_{k,\eta,Q}\right)\cup C\left(\overline{U_{k,\eta,Q}}\right),$
with $\max_{U_{k,\eta,Q}}\left|f_{k}\right|\leq\frac{1}{e^{3k}C_{4}^{k}}$
for any $k\geq k_{0}.$ By Lemma \ref{lem:Main estimate genrlize borel}
we have 
\[
\left|\left(\bbg f\right)^{\left(n\right)}\left(x\right)\right|\leq C^{n+1}\cdot n!\cdot m_{k}^{n/k}\cdot E\left(L\left(k\right)\right)\cdot E\left(x\right)\cdot\frac{1}{e^{3k}C_{4}^{k}},
\]
By Lemma \ref{lem:E exp bound}, $E\left(L\left(k\right)\right)\lesssim e^{2k}$,
and therefore
\[
\left|\left(\bbg f_{k}\right)^{\left(n\right)}\left(x\right)\right|\leq C^{n+1}\cdot n!\cdot m_{k}^{n/k}\cdot E\left(x\right)\cdot\frac{1}{e^{k}C_{4}^{k}}
\]
for any $n\geq0$. By the third assertion of Claim \ref{claim:m_k reg first claim},
\[
m_{k}^{n/k}\leq C_{4}^{k+n}m_{n},
\]
so, we conclude that 
\[
\left|\left(\bbg f_{k}\right)^{\left(n\right)}\left(x\right)\right|\leq\frac{C^{n+1}}{e^{k}}\cdot C_{4}^{n}\cdot M_{n}E\left(x\right).
\]
Thus, $\sum_{k\geq k_{0}}\bbg f_{k}$ converges to a function $G\in A(M,\gamma,J).$ Since $G$ and $\mathcal{B}_{\gamma}f$
have the same Taylor coefficients at the origin, we conclude the proof.
\end{proof}

\subsection{Estimates for the generalized Borel transform in the analytic case}

Throughout this section, we fix an admissible $\gamma$ with analytic
$\widehat{\gamma}$ (i.e. $\lim_{\rho\to\infty}\varepsilon\left(\rho\right)>0$),
and a regular sequence $M=\left(n!m_{n}\right)_{n\ge0}$. For a given
$\eta>0$, recall that 
\[
\Omega_{\eta}\left(\gamma\right):=\left\{ z:\sup_{t>0}E(t\eta)\left|K\left(\frac{t}{z}\right)\right|<\infty\right\} .
\]

In this section, we will typically suppress the dependence on $\gamma$,
writing $\Omega_{\eta}.$ For a given $k\in\mathbb{Z_{+}}$ and $Q>0$, we
consider the set 
\[
\Omega_{k,\eta,Q}=\left\{ z:\text{dist\ensuremath{\left(z,\Omega_{\eta}\right)}}\leq\frac{1}{Qm_{k}^{1/k}L(k)}\right\} 
\]
The next lemma is a variant of Lemma \ref{lem:Main estimate genrlize borel}
for the analytic case. 
\begin{lem}
\label{lem:Main estimate genrlize borel analytic}For any $1<\eta$,
$0<Q$, there exists a constant $C=C(\eta,Q)$, such that if $g\in\text{Hol}\left(\Omega_{k,\eta,Q}\right)\cap C\left(\overline{\Omega_{k,\eta,Q}}\right)$
then, for any $x\geq0$ and $n\in\mathbb{Z}_{+}$ we have
\[
\left|\left(\bbg g\right)^{(n)}\left(x\right)\right|\leq C^{n+1}\cdot n!\cdot m_{k}^{n/k}\cdot E\left(L\left(k\right)\right)E\left(x\eta\right).
\]
\end{lem}

\begin{proof}
Put $\lim_{x\to\infty}\widehat{\gamma}(x)^{1/x}=:\widehat{L}\in\left(\tfrac{1}{\pi},\infty\right)$.
Fix $1<\eta_{0}<\eta,$ so that $\frac{\eta_{0}}{\widehat{L}}<\pi.$
For $0<R<R^{\prime},$ consider the curve 
\[
\Gamma(R,R^{\prime}):=\left\{ Re^{i\theta}:\left|\theta\right|<\theta_{0}\right\} \cup\left\{ te^{\pm i\theta_{0}}:t\in\left[R,R^{\prime}\right]\right\} \cup\left\{ R^{\prime}e^{i\theta}:\theta_{0}<\left|\theta\right|\le\pi\right\} 
\]
where
\[
\theta_{0}=\frac{\eta_{0}}{\widehat{L}}.
\]
Put $G=\bbg g$. By Cauchy formula, 
\[
G\left(x\right)=\frac{1}{2\pi i}\int_{\Gamma(R,R^{\prime})}g\left(\frac{x}{w}\right)E\left(w\right)\frac{dw}{w},
\]
provided that
\[
\left\{ \frac{x}{w}:w\in\Gamma(R,R^{\prime})\right\} \subset \Omega_{k,\eta,C}.
\]
Moreover, by differentiating under the integral sign, we find that
\[
G^{\left(n\right)}\left(x\right)=\frac{1}{2\pi i}\int_{\Gamma(R,R^{\prime})}g^{\left(n\right)}\left(\frac{x}{w}\right)E\left(w\right)\frac{dw}{w^{n+1}}.
\]
Taking $R^{\prime}\to\infty,$ and putting
\[
\Gamma(R):=\left\{ Re^{i\theta}:\left|\theta\right|<\theta_{R}\right\} \cup\left\{ ue^{\pm i\theta_{R}}:u>R\right\} .
\]
We obtain
\[
G^{\left(n\right)}\left(x\right)=\frac{1}{2\pi i}\int_{\Gamma(R)}g^{\left(n\right)}\left(\frac{x}{w}\right)E\left(w\right)\frac{dw}{w^{n+1}}.
\]
By Cauchy's estimates if 
\begin{equation}
\left\{ \frac{x}{w}:w\in\Gamma(R)\right\} \subseteq\Omega_{\eta}\label{eq:inclu Gamma in V-1}
\end{equation}
then 
\[
\left|G^{\left(n\right)}\left(x\right)\right|\leq n!\left(Qm_{k}^{1/k}L(k)\right)^{n}\frac{1}{2\pi}\int_{\Gamma(R)}\left|E\left(w\right)\frac{dw}{w^{n+1}}\right|\leq n!\left(Qm_{k}^{1/k}L(k)\right)^{n}\frac{1}{R^{n}}\sup_{\Gamma(R)}\left|E\left(w\right)\right|.
\]
By Lemma \ref{lem: E curve lemma}, $\sup_{\Gamma(R)}\left|E\left(w\right)\right|\leq CE\left(R\right),$
so we obtain 
\[
\left|G^{\left(n\right)}\left(x\right)\right|\lesssim n!\left(Qm_{k}^{1/k}L(k)\right)^{n}\frac{E\left(R\right)}{R^{n}},
\]
provided that (\ref{eq:inclu Gamma in V-1}) holds. Finally we choose
\[
R=\max\left\{ L\left(k\right),\eta x\right\} ,
\]
and notice that 
\[
\left\{ \frac{x}{w}:w\in\Gamma(R)\right\} \subseteq\Omega_{\eta}.
\]
We have established
\[
\left|G^{\left(n\right)}\left(x\right)\right|\lesssim C^{n}\cdot n!\cdot m_{k}^{n/k}\max\left\{ E\left(L(k)\right),E\left(x\eta\right)\right\} ,
\]
as claimed.
\end{proof}
\begin{proof}[Proof of the inclusion $\mathcal{B}_{\gamma}D\left(M,\gamma;I\right)\subseteq A(M,\gamma,I)$
in the analytic case]
Without loss of generality, we assume that $I=[0,1)$ (otherwise
we re-scale). Let $\eta>1,$ $f\in C_{\infty}^{M\cdot\gamma}\left(\Omega_{\eta}\left(\gamma\right)\right)$,
and let $C_{4}>0$ be as in Claim \ref{claim:m_k reg first claim}.
by Lemma \ref{lem:aprox by analytic function-the analytic case}, for any closed $J\subseteq I$, 
there is a $Q>0$ and a sequence $\left(f_{k}\right)_{k\geq k_{0}}\subset C^{\infty}\left(J\right)$
so that $\sum_{k\geq k_{0}}f_{k}$ converges to $f$ in the $C^{\infty}\left(J\right)$
topology, and for any $k\geq k_{0},$ $f_{k}$ are analytic in $\left\{ z:\text{dist}\left(z,\text{\ensuremath{\Omega_{\eta}}}\mathbb{\left(\gamma\right)}\right)\leq\frac{1}{Q^{k}m_{k}^{1/k}L(k)}\right\} $
and satisfies therein $\max_{\Omega_{k,\eta,Q}}\left|f_{k}\right|\leq\frac{1}{C_{4}^{k}e^{3k}}$.
Notice that by Lemma \ref{lem:E exp bound}, for any $k$ sufficiently
big we have $E\left(L(k)\right)\leq e^{2k}.$ Thus, Lemma \ref{lem:Main estimate genrlize borel analytic}
yields
\[
\left|\left(\mathcal{B}_{\gamma}F_{k}\right)^{(n)}(x)\right|\lesssim\frac{1}{e^{k}C_{4}^{k}}C^{n}\cdot n!\cdot m_{k}^{n/k}E\left(x\eta\right),\quad\forall n\geq0,\,k\geq k_{0}.
\]
By the third assertion of Claim \ref{claim:m_k reg first claim},
\[
m_{k}^{n/k}\leq C_{4}^{k+n}m_{n}.
\]
We conclude that 
\[
\left|\left(\bbg F_{k}\right)^{\left(n\right)}\left(x\right)\right|\leq C^{n+1}\cdot C_{4}^{n}\cdot M_{n}E\left(x\eta\right).
\]
Since the right hand side is independent of $k$, we obtain
\[
\left|\left(\bbg f\right)^{\left(n\right)}\left(x\right)\right|\leq C^{n+1}\cdot C_{4}^{n}\cdot M_{n}E\left(x\eta\right).
\]
We have established the inclusion
\[
\mathcal{B}_{\gamma}C_{\infty}^{M\cdot\gamma}\left(\Omega_{\eta}\left(\gamma\right)\right)\subseteq A(M,\gamma,[0,\eta^{-1}]).
\]
Taking union over all $\eta>1$ yields $\mathcal{B}_{\gamma}D\left(M,\gamma;I\right)\subseteq A(M,\gamma,I).$
\end{proof}

\subsection{Estimates for the generalized Borel transform in $B^{+}\left(M\gamma,\widehat{\gamma},I\right)$}

Throughout this section, we fix an admissible $\gamma$ with non-analytic
$\widehat{\gamma}$ (i.e. $\lim_{\rho\to\infty}\varepsilon\left(\rho\right)=0)$,
a regular sequence $M=\left(n!m_{n}\right)_{n\geq 0}$ and put $J=[-1,1]$.
Recall that  $\gamma_{k}^{1/k}=L(k)$ and $\left(\widehat{\gamma}_{k}/k!\right)^{1/k}=\widehat{L}(k)$.
For a given $k\geq1$, $\eta>1$ and $Q>0$, the set $U_{k,\eta,Q}$
was defined as a $\frac{1}{QL(k)m_{k}^{1/k}}$ neighborhood of the
sectors
\[
\left\{ z:\left|\text{arg}(z)\right|\leq\frac{\eta}{\widehat{L}(k)},\,\quad\text{Re}\left(z\right)\in[0,1]\right\} \bigcup\left\{ z:\left|\text{arg}(-z)\right|\leq\frac{\eta}{\widehat{L}(k)},\,\quad\text{Re}\left(z\right)\in[-1,0]\right\} ,
\]
and that $S_{\mu}^{+}$ is the half disk $\left\{ z:\text{Im}z>0,\,|z|\leq\mu\right\} $.
Finally, put 
\[
U_{k,\eta,Q,\mu}^{+}=\left(\left\{ \left|\text{Re}\left(z\right)\right|\leq\mu\right\} \cap U_{k,\eta,Q}\right)\cup S_{\mu}^{+}.
\]

\begin{lem}
\label{lem:Main estimate genrlize borel class +}For $1<\eta$, $0<Q$
and $0<\mu<\frac{1}{2}$ there exists a $C>0$, such that if $g\in\text{Hol}\left(\overline{U_{k,\eta,Q,\mu}^{+}}\right)$
with $\Vert g\Vert_{L^{\infty}\left(U_{k,\eta,Q,\mu}^{+}\right)}\leq1,$
then $\bbg g\in\text{Hol}\left(\left\{ \mathcal{\text{Re}}\left(z\right)>0\right\} \right)\cap C^{\infty}\left(\left\{ \text{Re}\left(z\right)\geq0\right\} \right)$, 
\[
\left|\left(\bbg g\right)\left(z\right)\right|\lesssim E\left(\frac{2|z|}{\mu}\right),\quad\forall\mathcal{\mathcal{\text{Re}}}\left(z\right)\geq0,
\]
and
\[
\left|\left(\bbg g\right)^{(n)}\left(x\right)\right|\lesssim C^{n}\cdot n!\cdot m_{k}^{n/k}\max\left\{ E\left(L(k)\right),E\left(x\eta\right)\right\} \cdot E\left(\frac{2|x|}{\mu}\right),\quad\forall x\in\mathbb{R},\,n\geq0
\]
\end{lem}

\begin{proof}
Put $G=\mathcal{B}_{\gamma}g$. For $R>0$ put 
\[
\Gamma(R):=\left\{ Re^{i\theta}:\left|\theta\right|<\theta_{R}\right\} \cup\left\{ ue^{\pm i\theta_{R}}:u>R\right\} ,\quad\text{with\ensuremath{\quad}}\theta_{R}=\frac{\eta}{\widehat{L}\left(L^{-1}\left(R\right)\right)}.
\]
As in the proof of Lemma \ref{lem:Main estimate genrlize borel},
if
\[
\left\{ \frac{z}{w}:w\in\Gamma(R)\right\} \subset U_{k,\eta,Q,\mu}^{+},
\]
then 
\[
G\left(z\right)=\frac{1}{2\pi i}\int_{\Gamma(R)}g\left(\frac{z}{w}\right)E\left(w\right)\frac{dw}{w}.
\]
Thus, by Lemma \ref{lem: E curve lemma}, $\sup_{\Gamma(R)}\left|E\left(w\right)\right|\lesssim CE\left(R\right).$
Thus, the choice $R=\frac{|z|+1}{\mu}$ gives us 
\[
\left|G\left(z\right)\right|\lesssim E\left(\frac{|z|+1}{\mu}\right)\lesssim E\left(\frac{2|z|}{\mu}\right),
\]
for any $z\in\left\{ \text{Re}\left(z\right)\geq0\right\} .$ By Morera's
theorem, $G$ is analytic in $\left\{ \text{Re}\left(z\right)>0\right\} $.
Finally the inequality
\[
\left|G^{(n)}(x)\right|\leq C^{n}\cdot n!\cdot m_{k}^{n/k}\max\left\{ E\left(L(k)\right),E\left(\frac{2|x|}{\mu}\right)\right\} 
\]
is proven in the same way as in Lemma \ref{lem:Main estimate genrlize borel}.
\end{proof}
\begin{proof}[Proof of Lemma \ref{lem:intersection pm}]
Let $f_{+}\in B_{\eta}^{+}(M\gamma,\widehat{\gamma};0)$ and let
$C_{4}>0$ be as in Claim \ref{claim:m_k reg first claim}. By Lemma
\ref{lem:analytic aprox for =00005Cplus case} there are $Q>0$, $\mu>0$
and a sequence $\left(f_{k}\right)_{k\geq k_{0}}\subset C^{\infty}\left([-\mu,\mu]\right)$
so that $\sum_{k\geq k_{0}}f_{k}$ converges to $f_{+}$ in the $C^{\infty}\left([-\mu,\mu]\right)$
topology, and $f_{k}\in\text{Hol}\left(\overline{U_{k,\eta,Q,\mu}}\right)$
with $\max_{\overline{U_{k,\eta,Q,\mu}}}\left|f_{k}\right|\leq\min\left\{ \frac{1}{4^{k}C_{4}^{k}},\frac{1}{2^{k}}\right\} $
for any $k\geq k_{0}.$ By Lemma \ref{lem:Main estimate genrlize borel class +},
for each $k\geq k_{0}$, we have that $\bbg f_{k}\in\text{Hol}\left(\left\{ \text{Re}\left(z\right)>0\right\} \right)\cap C^{\infty}\left(\left\{ \text{Re}\left(z\right)\geq0\right\} \right)$. Therefore
\[
\left|\sum_{k\geq k_{0}}\left(\bbg f_{k}\right)\left(z\right)\right|\leq\sum_{k\geq k_{0}}\frac{C}{2^{k}}E\left(\frac{2|z|}{\mu}\right)\leq CE\left(\frac{2|z|}{\mu}\right),\quad\forall\text{Re}\left(z\right)\geq0.
\]
As in the proof of the inclusion $\bbg B_{1/\eta}(M\gamma,\widehat{\gamma};I)\subseteq A(M,\gamma,I)$
in the non-analytic case, we also have 
\[
\left|\sum_{k\geq k_{0}}\left(\bbg f_{k}\right)^{(n)}\left(x\right)\right|\lesssim C^{n}M_{n}E\left(\frac{2|x|}{\mu}\right),\quad\forall x\in\mathbb{R},n\geq0
\]
Observe that $\sum_{k\geq k_{0}}\bbg f_{k}$ and $\mathcal{B}_{\gamma}f_{+}$
are both elements of the quasianalytic class $C^{M}\left(\mathbb{R}\right)$
and have the same Taylor coefficients at the origin, thus by quasianalyticity
$\bbg f_{+}=\sum_{k\geq k_{0}}\bbg f_{k}.$ We have
established that for any $f_{+}\in B_{\eta}^{+}(M\gamma,\widehat{\gamma};0)$,
$\mathcal{B}_{\gamma}f_{+}\in\text{Hol}\left(\left\{ \text{Re}\left(z\right)>0\right\} \right)\cap C^{\infty}\left(\left\{ \text{Re}\left(z\right)\geq0\right\} \right)$
and there is a $\mu>0$ so that 
\[
\left|\left(\bbg f_{+}\right)\left(z\right)\right|\lesssim E\left(\frac{2|z|}{\mu}\right),\quad\forall\text{Re}\left(z\right)\geq0.
\]
Thus if $f\in B_{\eta}^{+}(M\gamma,\widehat{\gamma};0)\cap B_{\eta}^{-}(M\gamma,\widehat{\gamma};0),$
then by Morera's theorem, $\bbg f$ is entire and there is a $\mu>0$ such that
\[
\left|\left(\bbg f\right)\left(z\right)\right|\lesssim E\left(\frac{2|z|}{\mu}\right),\quad\forall z\in\mathbb{C}.
\]
By \cite[Theorem 7]{kiro2018taylor}, the latter implies that $f\in C^{\omega}(0)$
as claimed. 
\end{proof}

\section{Some Applications}

\subsection{Bijections onto Carleman classes. \label{subsec: log th}}

Let $\gamma$ be an admissible weight and let $M$ be a regular sequence.
If $M_{n}\gamma_{n+1}\lesssim C^{n}\widehat{\gamma}_{n}$ for a constant
$C>0$, then
\[
B_{1/\eta}(M\gamma,\widehat{\gamma};0_{+})=B_{\eta}(M\gamma,\widehat{\gamma};0_{+})=C_{\infty}^{M\gamma}\left(0_{+}\right).
\]
Therefore, in this case, the map
\begin{equation}
\lpg:A(M,\gamma;0_{+})\to C_{\infty}^{M\gamma}\left(0_{+}\right)\label{eq:L_gamma bijection}
\end{equation}
is \emph{onto}. Moreover, if in addition the class $C_{\infty}^{M\gamma}\left(0_{+}\right)$
is quasianalytic, then $\lpg$ is in fact a \emph{bijection}.

For example, if $M_{n}\approx n!\log^{n\alpha}\left(n\right)$ with
$0\leq\alpha\leq1,$ then (\ref{eq:L_gamma bijection}) is a bijection
provided that $xL^{\prime}(x)\lesssim\log^{-\alpha}x$, where $\gamma(x)=L^{x}(x)$.
The particular choice $\alpha=0$, i.e. $M_{n}=n!$ and $xL^{\prime}(x)\lesssim1$
is the aforementioned logarithmic threshold appearing in \cite[Corollary 1]{kiro2018taylor}.

It is also interesting to note that there are many ways to factor
a sequence $N$ as a product $M\gamma.$ For example, consider 
\[
M_{n}^{\alpha}\approx n!\log^{n\alpha}\left(n+e\right),\quad\text{and}\quad\gamma^{\alpha}(n)=\log^{n\alpha}\left(n+e\right).
\]
We have $M^{\alpha}\gamma^{1-\alpha}=M^{1}$, this implies  the following
family of bijections:
\[
\lpg_{\alpha}:A^{M^{1-\alpha}}(\gamma_{\alpha};I)\to C_{\infty}^{M^{1}}\left(I\right),\quad0<\alpha\leq1.
\]

\subsection{Multi-summability in Carleman classes}
While the $\gamma$-summation method is powerful, it is not universally applicable, particularly for functions in larger quasianalytic Carleman classes. Following the ideas of \'{E}calle and Ramis, we can overcome this limitation by iterating the summation process, a method known as multi-summability.

Let $\gamma$ be an admissible and put $M_{n}=n!\gamma\left(n+1\right).$
As we have already mentioned, a necessary and sufficient condition (up to regularity assumptions) for the $\left(\gamma\right)$--summation,
to sum all elements in the Carleman class $C^{M}(I)$ is that
$\gamma$ has at most logarithmic growth (i.e. $\gamma_{n}^{1/n}\lesssim\log n)$.
In this section, we will follow ideas by \'{E}calle \cite{ecalle1981fonctions}
and Ramis \cite{ramis1980series} in order to sum elements in larger
quasianalytic Carleman classes such as $M_{n}=n!\log^{n}\left(n+e\right)\log^{n}\log\left(n+e^{e}\right)$.

If $\gamma_{1}$, $\gamma_{2}$ are two admissible weights such that
$\gamma_{1}\cdot\gamma_{2}=\gamma$, then for a polynomial $P$, we
have 
\[
\lp_{\gamma_{1}}\lp_{\gamma_{2}}\bb_{\gamma}P=\lp_{\gamma}\bb_{\gamma}P=P,
\]
that is, $\lp_{\gamma_{1}}\lp_{\gamma_{2}}=\lp_{\gamma}$ for polynomials.
As it turns out, when we extend the domains of definition of these
two operators, the summation method $\lp_{\gamma_{1}}\lp_{\gamma_{2}}\bb_{\gamma}$
is always stronger than $\lp_{\gamma}\bb_{\gamma}$. Following \'{E}calle,
if $\lp_{\gamma_{1}}\lp_{\gamma_{2}}\bb_{\gamma}f=f$, we will say
that $f$ is $\left(\gamma_{1},\gamma_{2}\right)$\emph{--multi-summable}.
Similarly, given $k$ admissible functions $\gamma_{1},\cdots,\gamma_{k}$,
we say that $f$ is $\left(\gamma_{1},\cdots,\gamma_{k}\right)$\emph{--multi-summable}
if
\[
f=\lp_{\gamma_{k}}\cdots\lp_{\gamma_{1}}\bb_{\gamma_{1}\cdots\gamma_{k}}f.
\]

For example, consider 
\[
\gamma_{k}\left(n+1\right)=\log_{\left[k\right]}^{n}\left(n+\exp_{\left[k\right]}\left(1\right)\right),
\]
where $\log_{\left[k\right]}$ and $\exp_{\left[k\right]}$ are the
$k$-th iterates of the functions $x\mapsto\log x$ and $x\mapsto e^{x}$
respectively, and 
\[
M_{n}^{[k]}=n!\gamma_{1}\left(n+1\right)\gamma_{2}\left(n+1\right)\cdots\gamma_{k}\left(n+1\right),\quad M_{n}^{[0]}=n!.
\]
A simple computation yields 
\[
\widehat{\gamma}_{k}\left(n+1\right)\sim\left(\frac{2}{\pi}+o(1)\right)^{n}M_{n}^{[k]}.
\]
By Theorem \ref{thm:Main thm}, for any $\eta>1$ we have 
\[
B_{1/\eta}\left(M^{[k-1]}\gamma_{k},\widehat{\gamma}_{k};I\right)\subseteq\lp_{\gamma_{k}}A^{M^{[k-1]}}\text{\ensuremath{\left(\gamma_{k};I\right)}}\subseteq B_{\eta}\left(M^{[k-1]}\gamma_{k},\widehat{\gamma}_{k};I\right).
\]
Since 
\[
M_{n}^{[k-1]}\gamma_{k}\left(n+1\right)=M_{n}^{[k]}=\left(\frac{\pi}{2}+o(1)\right)^{n}\widehat{\gamma}_{k}\left(n+1\right),
\]
we have
\[
B_{1/\eta}\left(M^{[k-1]}\gamma_{k},\widehat{\gamma}_{k};I\right)=B_{\eta}\left(M^{[k-1]}\gamma_{k},\widehat{\gamma}_{k};I\right)=C_{\infty}^{M^{[k]}}\left(I\right).
\]
We conclude that 
\[
\bb_{\gamma_{k}}:C_{\infty}^{M^{[k]}}\left(I\right)\to A\left(M^{[k-1]},\gamma_{k};I\right)
\]
is a bijection with inverse $\lp_{\gamma_{k}}.$ By definition, $A\left(M^{[k-1]},\gamma_{k};I\right)\subseteq C_{\infty}^{M^{[k-1]}}\left(\mathbb{R}_{+}\right)$,
and therefore
\[
f=\lp_{\gamma_{k}}\bb_{\gamma_{k}}f=\lp_{\gamma_{k}}\lp_{\gamma_{k-1}}\bb_{\gamma_{k-1}}\bb_{\gamma_{k}}f=\cdots=\lp_{\gamma_{k}}\cdots\lp_{\gamma_{1}}\bb_{\gamma_{1}}\cdots\bb_{\gamma_{k}}f=\lp_{\gamma_{k}}\cdots\lp_{\gamma_{1}}\bb_{\gamma_{1}\cdots\gamma_{k}}f,
\]
for any $f\in C_{\infty}^{M^{[k]}}\left(I\right)$. We have established the following:
\begin{prop}
With the notation above, any function in $C_{\infty}^{M^{[k]}}\left(I\right)$
is $\left(\gamma_{1},\cdots,\gamma_{k}\right)$--multi-summable.
\end{prop}

The above idea can be generalized further: given a sequence $M$ such
that $\sum M_{n}/M_{n+1}=\infty$, define the sequences $M^{(k)}$
and $\gamma_{k}$ by the following relations
\[
\gamma_{1}\left(n+1\right)^{\frac{1}{n}}=\prod_{j=1}^{n}\left(1-\frac{M_{n}}{M_{n+1}}\right)^{-1},\quad M=M^{(1)}\gamma_{1},
\]
\[
\gamma_{k}\left(n+1\right)^{\frac{1}{n}}=\prod_{j=1}^{n}\left(1-\frac{M_{n}^{(k-1)}}{M_{n+1}^{(k-1)}}\right)^{-1},\quad M^{\left(k-1\right)}=M^{(k)}\gamma_{k}.
\]
Note that if $M_{n}=n!\gamma(n+1)$, with admissible $\gamma$, then
$\gamma_{1},\cdots,\gamma_{k}$ are all admissible, and that by Theorem  \ref{thm:Main thm}, the maps 
\[
\bb_{\gamma_{k}}:C_{\infty}^{M^{(k-1)}}\left(I\right)\to A\left(M^{[k-1]},\gamma_{k};I\right),\quad M^{(0)}=M
\]
are bijections. So, the same reasoning yields that any $f\in C_{\infty}^{M}\left(I\right)$
is $\left(\gamma_{k},\cdots,\gamma_{1}\right)$-multi-summable.

\subsection{Resurgent theory}

Resurgent theory was introduced by \'{E}calle to extend the classical
Borel-Laplace summation by allowing singularities in the Laplace plane.
In resurgent theory, we do not associate to $\bb\lp f$ a formal power
series, but rather a more complicated object called a transseries
(see \cite{ecalle1981fonctions}). In recent years resurgent theory has
found many applications in models related to QFT and other fields
of mathematical physics. Central to these applications are functions
$f$ having a transseries asymptotic expansion of the form
\begin{equation}
f(x)\sim\sum_{j=0}^{N}\sum_{n=0}^{\infty}e^{-\frac{a_{j}}{x}}x^{n}c_{n,j},\,\,\text{as}x\to0^{+}\quad0\leq a_{0}<a_{1}<\cdots<a_{N}.\label{eq:Transseries}
\end{equation}
In the classical Borel--Laplace summation, these expansions arise
from the discontinuity of the function $\mathcal{B}f$ at the points $\left(a_{j}\right)$.
Indeed, for $a\geq0$, consider the shift operator 
\[
\left(\tau_{a}F\right)(x)=\begin{cases}
F(x-a), & x\geq a\\
0, & x<a,
\end{cases}
\]
and observe that 
\[
\mathcal{L}(\tau_{a}F)(x)=\frac{1}{x}\int_{a}^{\infty}F(t-a)e^{-t/x}dt=\frac{e^{-\frac{a}{x}}}{x}\int_{0}^{\infty}F(u)e^{-u/x}du=e^{-\frac{a}{x}}\cdot\left(\left(\mathcal{L}F\right)(x)\right).
\]
Therefore, if 
\[
F=F_{0}+\cdots+F_{N}\in\left(\tau_{a_{0}}A\left(M,\gamma;[0,1)\right)+\cdots+\tau_{a_{0}}A\left(M,\gamma;[0,1)\right)\right),
\]
then by Theorem \ref{thm:Main thm analytic}, 
\[
\mathcal{L}F\in\left(e^{-\frac{a_{0}}{x}}D\left(M,\Gamma;[0,1)\right)+\cdots+e^{-\frac{a_{N}}{x}}D\left(M,\Gamma;[0,1)\right)\right).
\]
Thus, the function $\mathcal{L}F$ has an expansion of the form (\ref{eq:Transseries})
with 
\[
c_{n,j}=F_{j}^{(n)}\left(a_{j}\right).
\]

The application of Theorem \ref{thm:Main thm analytic} gives us the following
proposition: 
\begin{prop}
\label{prop:Resurgent}Let $0\leq a_{0}<a_{1}<\cdots<a_{N}$ and let
$M$ be a regular sequence so that $\sum M_{n}/M_{n+1}=\infty.$ The
Laplace transform 
\[
\mathcal{L}:\tau_{a_{0}}A\left(M,\gamma;[0,1)\right)\oplus\cdots\oplus\tau_{a_{N}}A\left(M,\gamma;[0,1)\right)\to e^{-\frac{a_{0}}{x}}D\left(M,\Gamma;[0,1)\right)\oplus\cdots\oplus e^{-\frac{a_{N}}{x}}D\left(M,\Gamma;[0,1)\right)
\]
is a \emph{bijection}.
\end{prop}

\begin{proof}
Since $\mathcal{L}\text{\ensuremath{\tau_{a}F}}=e^{-\frac{a}{x}}\cdot\mathcal{L}F,$
Theorem \ref{thm:Main thm analytic} yields that $\mathcal{L}:\tau_{a_{j}}A\left(M,\gamma;[0,1)\right)\to e^{-\frac{a_{j}}{x}}D\left(M,\Gamma;[0,1)\right)$
is a bijection. So, we only need to show that the two sums appearing
in the statement are direct. Let $F=F_{0}+\cdots+F_{N}\in\left(\tau_{a_{0}}A\left(M,\gamma;[0,1)\right)+\cdots+\tau_{a_{0}}A\left(M,\gamma;[0,1)\right)\right),$
and observe that 
\[
F_{0}^{(n)}(a_{0})=F^{(n)}(a_{0}),
\]
and
\[
F_{j}^{(n)}(a_{j})=F^{(n)}(a_{j})-F_{0}^{(n)}(a_{j})-\cdots-F_{j-1}^{(n)}(a_{j})
\]
Thus, by quasianalyticity, first $F$ recursively determines $F_{0}, F_1,\cdots$ uniquely. We conclude that the sum
$A\left(M,\gamma;[0,1)\right)+\cdots+\tau_{a_{N}}A\left(M,\gamma;[0,1)\right)$
is indeed direct. Similarly if $f=f_{0}+\cdots f_{N}\in\left(e^{-\frac{a_{0}}{x}}D\left(M,\Gamma;[0,1)\right)+\cdots+e^{-\frac{a_{N}}{x}}D\left(M,\Gamma;[0,1)\right)\right)$
then $e^{\frac{a_{0}}{x}}f(x)$ and $e^{\frac{a_{0}}{x}}f_{0}(x)$
have the same power series expansion at the origin, so $f$ uniquely determines
$f_{0}$. Continuing in the same way, we conclude that the
second sum is also direct. 
\end{proof}

\subsection{Euler type differential equations}

In this section we fix an admissible weight $\gamma,$ $I=\left[0,1\right)$
and $\eta>1$. Consider the operator $V=V_{\gamma}:\mathbb{C}\left[x\right]\to\mathbb{C}\left[x\right]$,
defined by 
\[
Vx^{n}=\frac{\gamma(n+2)}{\gamma(n+1)}x^{n+1}.
\]
For example, if $\gamma=\Gamma$, the Euler Gamma function, then $Vf(x)=x^{2}f^{\prime}(x)$.
Another example is 
\[
\gamma(n+1)=\frac{\Gamma(1)}{\Gamma(1-a)}\cdots\frac{\Gamma(n+1)}{\Gamma(n+1-a)},\quad0<a<1,
\]
so $V$ is a fractional derivative operator:
\[
Vf(x)=x^{1+a}\frac{d^{a}}{dx^{a}}f(x).
\]

\begin{prop}
Let $\gamma$ be admissible and $M$ be a regular quasianalytic sequence.
If $g\in B_{1/\eta}\left(M\gamma,\widehat{\gamma};0_{+}\right)$, and
$P$ is a polynomial without zeros in $[0,\infty)$, then the equation
\begin{equation}
P\left(V_{\gamma}\right)f=g\label{eq: Euler type}
\end{equation}
has a solution $f=\lp_{\gamma}\left(\frac{\bbg g}{P}\right)\in B_{\eta}\left(M\gamma,\widehat{\gamma};I\right)$.
\end{prop}

\begin{proof}
Since 
\[
\bb_{\gamma}V\left(x^{n}\right)=\bbg\left(\frac{\gamma(n+2)}{\gamma(n+1)}x^{n+1}\right)=\frac{x^{n+1}}{\gamma(n+1)}=x\bb_{\gamma}\left(x^{n}\right),
\]
we have (at least formally), 
\[
\bb_{\gamma}\left(P\left(V\right)f\right)=P\cdot\bb_{\gamma}f.
\]
By Theorem \ref{thm:Main thm}, 
\[
\bb_{\gamma}g\in A\left(M,\gamma;I\right).
\]
Since $P$ has no zeros in $[0,\infty)$, we have 
\[
\left\Vert \frac{d^{n}}{dx^{n}}\frac{1}{P}\right\Vert _{L^{\infty}\left(\mathbb{R}_{+}\right)}\leq C^{n}n!.
\]
\end{proof}
Thus $\frac{1}{P}\in A\left(M,\gamma;I\right),$ and 
\[
\frac{1}{P}\bb_{\gamma}g\in A\left(M,\gamma;I\right).
\]
By the same theorem, we also have 
\[
f:=\lp_{\gamma}\left(\frac{\mathcal{B}_{\gamma}g}{P}\right)\in\lpg A\left(M,\gamma;I\right)\subset B_{\eta}\left(M\gamma,\widehat{\gamma};I\right).
\]
Since $f$ is a formal solution of the equation \ref{eq: Euler type},
and the class $\lpg A\left(M,\gamma;I\right)$ is quasianalytic,
the function $f$ is in fact a solution to \ref{eq: Euler type} as
claimed. 
\begin{rem}
The assumption that $P$ is a polynomial without zeros on the positive ray can be weakened. The same statement is true for any function $P$
such that $\frac{1}{P}\in A\left(M,\gamma;\mathbb{R}_{+}\right).$
\end{rem}

\appendix

\section{Lemmas on the functions $K$ and $E$}
\begin{proof}[Proof of Lemma \ref{lem:three E inq}]
By Theorems A and B
\[
\log E(r)\sim\log\frac{1}{K(r)}\sim\rho\varepsilon(\rho)
\]
where $\rho=\rho(r)$ is the solution of $r=L(\rho)e^{\varepsilon(\rho)}.$
Thus, for any $\delta_{1}\in(0,1)$, it is sufficient to find a $\delta>0$
such that 
\[
\rho(\delta r)\varepsilon\left(\rho(\delta r)\right)+\rho(\eta r)\varepsilon\left(\rho(\eta r)\right)-\rho(r)\varepsilon\left(\rho(r)\right)\lesssim1.
\]
Thus, it is sufficient to show that there exists $\alpha>0$ (independent
of $\eta$), such that 
\[
\rho\left(\eta r\right)\varepsilon\left(\rho\left(\eta r\right)\right)\leq\eta^{\alpha}\rho(r)\varepsilon\left(\rho(r)\right).
\]
Let us prove the last assertion.
\begin{align*}
\log\frac{\rho(r)\varepsilon\left(\rho(r)\right)}{\rho(\eta r)\varepsilon\left(\rho(\eta \rho\right)} & =\int_{\eta r}^{r}\frac{\varepsilon\left(\rho(u)\right)+\rho(u)\varepsilon'\left(\rho(u)\right)}{\rho(u)\varepsilon\left(\rho(u)\right)}\rho'(u)du\\
& =(1+o(1))\int_{\eta r}^{r}\frac{\rho'(u)}{\rho(u)}du,\quad r\to\infty,
\end{align*}
where in the last equality, we have used $\rho|\varepsilon'(\rho)|=o\left(\varepsilon(\rho)\right)$.
Differentiating $\log r=\log L(\rho)+\varepsilon(\rho)$ yields 
\[
\frac{\rho'(r)}{\rho(r)}=\frac{1}{r}\cdot\frac{1}{\varepsilon\left(\rho(r)\right)+\rho(r)\varepsilon'\left(\rho(r)\right)}=\frac{1}{r\varepsilon\left(\rho(r)\right)}\left(1+o(1)\right).
\]
So
\[
\log\frac{\rho(r)\varepsilon\left(\rho(r)\right)}{\rho(\eta r)\varepsilon\left(\rho\left(\eta r\right)\right)}=(1+o(1))\int_{\eta r}^{r}\frac{1}{\varepsilon\left(\rho(u)\right)}\frac{du}{u}.
\]
Since $\varepsilon$ is bounded above and positive, there exists
$\alpha>0$, such that
\[
\log\frac{\rho(r)\varepsilon\left(\rho(r)\right)}{\rho(\eta r)\varepsilon\left(\rho\left(\eta r\right)\right)}>\alpha\log\frac{1}{1-\delta}.
\]
The last inequality completes the proof.
\end{proof}
\begin{proof}[Proof of Lemma \ref{lem: E curve lemma}]
Let $\delta>0,$ we assume that $z=re^{i\psi}$, $s=\rho e^{i\theta}$,
are related through the saddle-point equation
\[
\log z=\log L(s)+\varepsilon(s).
\]
Comparing the real and imaginary parts of the saddle-point equation
and making use of (\ref{eq:gamma asymp full}), we find that
\[
r=L(\rho)e^{\varepsilon(\rho)}(1+o(1)),\quad\theta\varepsilon(\rho)(1+o(1))=\psi,\quad r\to\infty.
\]
By Theorem B, $\left|zE(z)\right|=O(1)$, as $|z|\to\infty$, uniformly
in the set $\mathbb{C}\setminus\Omega\left(\tfrac{\pi}{2}+\delta\right)$.
Since, 
\[
\partial\Omega\left(\eta\tfrac{\pi}{2}\right):=\left\{ z\;:\;\theta=\pm\left(\frac{\pi}{2}+\delta\right),\;\rho>\rho_{0}\right\} ,
\]
it suffices to show that
\[
\eta\varepsilon\left(L^{-1}\left(r\right)\right)\geq\varepsilon\left(\rho\right)
\]
for $r$ sufficiently large.

If $\lim_{\rho\to\infty}\varepsilon(\rho)>0$, the statement
is trivial because $\eta>1$. Otherwise, $\lim_{\rho\to\infty}\varepsilon(\rho)=0$
and $\varepsilon$ is eventually non--increasing.

Write 
\[
\log\frac{\varepsilon(\rho)}{\varepsilon\left(L^{-1}(r)\right)}=-\int_{\rho}^{\rho e^{\varepsilon\left(L^{-1}(\rho)\right)}(1+o(1))}\frac{\varepsilon'(u)}{\varepsilon(u)}du.
\]
By assumption $-\frac{\varepsilon'(u)}{\varepsilon(u)}\leq\frac{1}{u}$
for $u$ sufficiently large, thus
\[
\log\frac{\varepsilon(\rho)}{\varepsilon\left(L^{-1}(r)\right)}\leq\varepsilon\left(L^{-1}(\rho)\right)\to_{r\to\infty}0,
\]
which completes the proof of Lemma \ref{lem: E curve lemma}.
\end{proof}
\begin{proof}[Proof of Lemma \ref{lem:E exp bound}]
By Theorem B,
\[
\log\left(rE(r)\right)\sim s\varepsilon(s)
\]
where $L(s)e^{\varepsilon(s)}=r.$ Put $r=L(k).$ Since $\varepsilon$
is non-negative and slowly varying, $L$ and $s\mapsto s\varepsilon(s)$
are eventually non-decreasing. It follows from $L(s)e^{\varepsilon(s)}=L(k)$, that
$s\leq k$, and therefore
\[
\log\left(rE(r)\right)\leq k\varepsilon\left(k\right)\left(1+o\left(1\right)\right).
\]
By assumption, $\lim_{k\to\infty}\varepsilon(k)<2$, which implies
\[
rE(r)\leq Ce^{2k},
\]
as claimed.

\end{proof}
\bibliographystyle{plain}
\nocite{*}
\bibliography{bib}

@article {KIROSODIN,
    AUTHOR = {Kiro, A. and Sodin, M.},
     TITLE = {On functions {$K$} and {$E$} generated by a sequence of
              moments},
   JOURNAL = {Expo. Math.},
  FJOURNAL = {Expositiones Mathematicae},
    VOLUME = {35},
      YEAR = {2017},
    NUMBER = {4},
     PAGES = {443--477},
}

@article{mandelbrojt1942classes,
  title={Classes of infinitely differentiable functions},
  author={Mandelbrojt, S.},
  journal={Rice Institute Pamphlet-Rice University Studies},
  volume={29},
  number={1},
  year={1942},
  publisher={Rice University}
}

@article{moroz1990summability,
  title={Summability method for a Horn-Shaped region},
  author={Moroz, A.},
  journal={Communications in mathematical physics},
  volume={133},
  number={2},
  pages={369--382},
  year={1990},
  publisher={Springer}
}

@article{nazarov2004lower,
  title={Lower bounds for quasianalytic functions, I. How to control smooth functions},
  author={Nazarov, F. and Sodin, M. and Volberg, A.},
  journal={Mathematica Scandinavica},
  pages={59--79},
  year={2004},
  publisher={JSTOR}
}

@book{seneta2006regularly,
  title={Regularly varying functions},
  author={Seneta, E.},
  volume={508},
  year={2006},
  publisher={Springer}
}

@article{moroz1990novel,
  title={Novel summability methods generalizing the Borel method},
  author={Moroz, A.},
  journal={Czechoslovak Journal of Physics},
  volume={40},
  number={7},
  pages={705--726},
  year={1990},
  publisher={Springer}
}

@inproceedings{korenblum1966non,
  title={Non-triviality conditions for certain classes of functions analytic in an angle and problem of quasianalyticity},
  author={Korenblum, B.},
  booktitle={Doklady Akademii Nauk},
  volume={166},
  number={5},
  pages={1046--1049},
  year={1966},
  organization={Russian Academy of Sciences}
}

@book{bingham1989regular,
  title={Regular variation},
  author={Bingham, N. and Goldie, C. and Teugels, J.},
  volume={27},
  year={1989},
  publisher={Cambridge university press}
}

@article{kiro2018taylor,
  title={On Taylor coefficients of smooth functions},
  author={Kiro, A.},
  journal={Accepted for publication in Journal d'Analyse Mathematique},
  year={2019}
}

@article{buhovsky2019power,
  title={Power substitution in quasianalytic Carleman classes},
  author={Buhovsky, L. and Kiro, A. and Sodin, S.},
  journal={Israel Journal of Mathematics},
  pages={1--12},
  year={2019},
  publisher={Springer}
}

@book{beurling1989collected,
  title={The collected works of Arne Beurling: Complex analysis, edited by L. Carleson,  P. Malliavin, and J. Neuberger and J. Wermer},
  author={Beurling, A.},
  volume={1},
  year={1989},
  publisher={Birkhauser}
}

@book{carleman1926fonctions,
  title={Les Fonctions quasi analytiques: le{\c{c}}ons profess{\'e}es au College de France},
  author={Carleman, T.},
  year={1926},
  publisher={Gauthier-Villars et Cie}
}

@article{carleson1961universal,
  title={On universal moment problems},
  author={Carleson, L.},
  journal={Mathematica Scandinavica},
  volume={9},
  number={1b},
  pages={197--206},
  year={1961},
  publisher={JSTOR}
}

@book{costin,
  title={Asymptotics and Borel summability},
  author={Costin, Ovidiu},
  year={2008},
  publisher={Chapman and Hall/CRC}
}

@book{nevanlinna1918theorie,
  title={Zur Theorie der asymptotischen Potenzreihen I. II.},
  author={Nevanlinna, F.},
  year={1918},
  publisher={Suomalaisen tiedeakademian kustantama}
}

@article{dyn1976pseudoanalytic,
  title={Pseudoanalytic continuation of smooth functions. Uniform scale},
  author={Dynkin, E. M. },
  journal={Mathematical programming and related questions (Proc. Seventh Winter School, Drogobych, 1974), Theory of functions and functional analysis (Russian)},
  pages={40--73},
  year={1976}
}

@book{ehrenpreis2011fourier,
  title={Fourier analysis in several complex variables},
  author={Ehrenpreis, L.},
  year={2011},
  publisher={Courier Corporation}
}

@book{hardy2000divergent,
  title={Divergent series},
  author={Hardy, G. H.},
  volume={334},
  year={2000},
  publisher={American Mathematical Soc.}
}

@article{sokal1980improvement,
  title={An improvement of Watson's theorem on Borel summability},
  author={Sokal, A. D.},
  journal={Journal of Mathematical Physics},
  volume={21},
  number={2},
  pages={261--263},
  year={1980},
  publisher={American Institute of Physics}
}

@book{goldberg2008value,
  title={Value distribution of meromorphic functions},
  author={Goldberg, A. and Ostrovskii, I.},
  volume={236},
  year={2008},
  publisher={American Mathematical Soc.}
}

@article{lastra2015summability,
  title={Summability in general Carleman ultraholomorphic classes},
  author={Lastra, A. and Malek, S. and Sanz, J.},
  journal={Journal of Mathematical Analysis and Applications},
  volume={430},
  number={2},
  pages={1175--1206},
  year={2015},
  publisher={Elsevier}
}

@article{michalik2012multisummability,
  title={Multisummability of formal solutions of inhomogeneous linear partial differential equations with constant coefficients},
  author={Michalik, S.},
  journal={Journal of dynamical and control systems},
  volume={18},
  number={1},
  pages={103},
  year={2012},
  publisher={Citeseer}
}

@article{michalik2019stokes,
  title={The Stokes phenomenon for some moment partial differential equations},
  author={Michalik, S. and Tkacz, B.},
  journal={Journal of Dynamical and Control Systems},
  volume={25},
  number={4},
  pages={573--598},
  year={2019},
  publisher={Springer}
}

@article{balser2010gevrey,
  title={Gevrey order of formal power series solutions of inhomogeneous partial differential equations with constant coefficients},
  author={Balser, W. and Yoshino, M.},
  journal={Funkcialaj Ekvacioj},
  volume={53},
  number={3},
  pages={411--434},
  year={2010},
  publisher={Division of Functional Equations, The Mathematical Society of Japan}
}

@article{balser1997moment,
  title={Moment methods and formal power series},
  author={Balser, W.},
  journal={Journal de math{\'e}matiques pures et appliqu{\'e}es},
  volume={76},
  number={3},
  pages={289--305},
  year={1997},
  publisher={Elsevier Masson}
}

@incollection{sanz2017asymptotic,
  title={Asymptotic analysis and summability of formal power series},
  author={Sanz, J.},
  booktitle={Analytic, Algebraic and Geometric Aspects of Differential Equations},
  pages={199--262},
  year={2017},
  publisher={Springer}
}

@book{ramis1993series,
  title={S{\'e}ries divergentes et th{\'e}ories asymptotiques},
  author={Ramis, J.-P.},
  volume={121},
  year={1993},
  publisher={Soci{\'e}t{\'e} math{\'e}matique de France}
}

@article{carleson1952sets,
  title={Sets of uniqueness for functions regular in the unit circle},
  author={Carleson, L.},
  journal={Acta Mathematica},
  volume={87},
  pages={325--345},
  year={1952},
  publisher={Institut Mittag-Leffler}
}

@inproceedings{korenblum1965quasianalytic,
  title={Quasianalytic classes of functions in a circle},
  author={Korenblum, B.},
  booktitle={Doklady Akademii Nauk},
  volume={164},
  number={1},
  pages={36--39},
  year={1965},
  organization={Russian Academy of Sciences}
}

@article{salinas1955funciones,
  title={Funciones con momentos nulos},
  author={Salinas, B.R.},
  journal={Rev. Acad. Ci. Madrid},
  volume={49},
  pages={331--368},
  year={1955}
}

@article{hirschman1950behaviour,
  title={On the behaviour of Fourier transforms at infinity and on quasi-analytic classes of functions},
  author={Hirschman, I. I.},
  journal={American Journal of Mathematics},
  volume={72},
  number={1},
  pages={200--213},
  year={1950},
  publisher={JSTOR}
}

@book{balser2008formal,
  title={Formal power series and linear systems of meromorphic ordinary differential equations},
  author={Balser, W.},
  year={2008},
  publisher={Springer Science \& Business Media}
}

@article{aniceto2019primer,
  title={A primer on resurgent transseries and their asymptotics},
  author={Aniceto, I. and Ba{\c{s}}ar, G. and Schiappa, R.},
  journal={Physics Reports},
  volume={809},
  pages={1--135},
  year={2019},
  publisher={Elsevier}
}

@article{miravitllas2019resurgence,
  title={Resurgence, a problem of missing exponential corrections in asymptotic expansions},
  author={Miravitllas Mas, R.},
  journal={arXiv},
  pages={arXiv--1904},
  year={2019}
}

@incollection{ramis1980series,
  title={Les series k--sommables et leurs applications},
  author={Ramis, J.-P.},
  booktitle={Complex analysis, microlocal calculus and relativistic quantum theory},
  pages={178--199},
  year={1980},
  publisher={Springer}
}

@article{Ramis_and_Sibuya ,
     author = {Ramis, J.-P. and Sibuya, Y.},
     title = {A new proof of multisummability of formal solutions of non linear meromorphic differential equations},
     journal = {Annales de l'Institut Fourier},
     publisher = {Association des Annales de l'institut Fourier},
     volume = {44},
     number = {3},
     year = {1994},
     pages = {811-848},
}

@article{ecalle1993cohesive,
  title={Cohesive functions and weak accelerations},
  author={Ecalle, J.},
  journal={Journal d'Analyse Mathematique},
  volume={60},
  number={1},
  pages={71--97},
  year={1993},
  publisher={Springer}
}

@book{ecalle1981fonctions,
  title={Les fonctions r{\'e}surgentes:(en trois parties)},
  author={{\'E}calle, J.},
  volume={1},
  year={1981},
  publisher={Universit{\'e} de Paris-Sud, D{\'e}partement de Math{\'e}matique, B{\^a}t. 425}
}

@book{ecalle1985fonctions,
  title={Les fonctions resurgentes: L'{\'e}quation du pont et la classification analytique des objets locaux},
  author={{\'E}calle, J.},
  volume={3},
  year={1985},
  publisher={Universit{\'e} de Paris-Sud, D{\'e}partement de Math{\'e}matique, B{\^a}t. 425}
}

@article{dudko2015resurgent,
  title={On the resurgent approach to {\'E}calle--Voronin's invariants},
  author={Dudko, A. and Sauzin, D.},
  journal={Comptes Rendus Mathematique},
  volume={353},
  number={3},
  pages={265--271},
  year={2015},
  publisher={Elsevier}
}

@article{sauzin2014introduction,
  title={Introduction to 1-summability and resurgence},
  author={Sauzin, D.},
  journal={arXiv preprint arXiv:1405.0356},
  year={2014}
}

@article{zinn1981perturbation,
  title={Perturbation series at large orders in quantum mechanics and field theories: application to the problem of resummation},
  author={Zinn-Justin, J.},
  journal={Physics Reports},
  volume={70},
  number={2},
  pages={109--167},
  year={1981},
  publisher={Elsevier}
}

@article{aniceto2015resurgent,
  title={Resurgent analysis of localizable observables in supersymmetric gauge theories},
  author={Aniceto, I. and Russo, J. and Schiappa, R.},
  journal={Journal of High Energy Physics},
  volume={2015},
  number={3},
  pages={172},
  year={2015},
  publisher={Springer}
}

@article{aniceto2011resurgence,
  title={The resurgence of instantons in string theory},
  author={Aniceto, I. and Schiappa, R. and Vonk, M.},
  journal={arXiv preprint arXiv:1106.5922},
  year={2011}
}

\end{document}